\documentclass[a4paper,11pt,leqno]{amsart}
\usepackage{bbm}
\usepackage{amsmath,amssymb,amsthm,euro} 
\usepackage{bbm}
\usepackage{stmaryrd}
\usepackage{tensor}
\usepackage{enumerate}
\usepackage{todonotes}
\usepackage{verbatim}
\usepackage{hyperref}

\newtheorem{theorem}{Theorem}

\newtheorem{definition}[theorem]{Definition}
\newtheorem{lemma}[theorem]{Lemma}
\newtheorem{scholium}[theorem]{Scholium}
\newtheorem{proposition}[theorem]{Proposition}

\newtheorem{example}[theorem]{Example}
\numberwithin{theorem}{section}

\newcommand\8{\infty}
\renewcommand\t{\tau}
\newcommand\e{\varepsilon}
\newcommand\s{\sigma}
\renewcommand\o{\omega}
\newcommand\st{ [\s,\t ] }

\numberwithin{equation}{section}

\newcommand\eps{{\epsilon}}
\newcommand\bE{{\mathbb E}}

\newcommand\bL{{\ell}}

\newcommand\bN{{\mathbb N}}
\newcommand\bP{{\mathbb P}}
\newcommand\bQ{{\mathbb Q}}
\newcommand\bR{{\mathbb R}}
\newcommand\bZ{{\mathbb Z}}

\newcommand\cB{\mathcal{B}}

\newcommand\cF{\mathcal{F}}

\newcommand\cL{\mathcal{L}}

\newcommand\cP{\mathcal{P}}
\newcommand\cQ{\mathcal{Q}}
\newcommand\cS{\mathcal{S}}

\newcommand\p{^\prime}
\newcommand\pp{^{\prime\prime}}

\newcommand\bstar{\begin{eqnarray*}}
\newcommand\estar{\end{eqnarray*}}
\newcommand\be{\begin{equation}}
\newcommand\ee{\end{equation}}
\newcommand\bea{\begin{eqnarray}}
\newcommand\eea{\end{eqnarray}}
\newcommand\1{{\bf 1}}

\newcommand\ovl{\overline}
\newcommand\udl{\underline}

\usepackage{color}

\title{Pathwise Stochastic Calculus with Local
Times}
\author{Mark Davis, Jan Ob\l\'oj and Pietro Siorpaes}
\address{Davis: Department of Mathematics, Imperial College London, London SW7 2AZ}
\email{mark.davis@imperial.ac.uk}
\urladdr{http://www.imperial.ac.uk/people/mark.davis}
\address{Ob\l\'oj, Siorpaes: Mathematical Institute, University of Oxford, Oxford OX2 6GG}
\email{obloj@maths.ox.ac.uk, siorpaes@maths.ox.ac.uk}
\urladdr{www.maths.ox.ac.uk/people/jan.obloj, www.maths.ox.ac.uk/people/pietro.siorpaes}
\date{First version: 16 May 2012. This version: \today{}.}

\begin{document}
\begin{abstract}
We study a notion of local time for a continuous path, defined as a limit
of suitable discrete quantities along a general sequence of partitions of
the time interval. Our approach subsumes other existing definitions and
agrees with the usual (stochastic) local times a.s. for paths of a
continuous semimartingale. We establish pathwise version of the
Tanaka-Meyer, change of variables and change of time formulae. We provide
equivalent conditions for existence of pathwise local time. Finally, we study in detail how the limiting objects,
the quadratic variation and the local time, depend on the choice of
partitions. In particular, we show that an arbitrary given non-decreasing
process can be achieved a.s.\ by the pathwise quadratic variation of a
standard Brownian motion for a suitable sequence of (random) partitions; however, such degenerate behaviour is excluded when the partitions are constructed from stopping times.
\end{abstract}
\thanks{This research has been generously supported by the European Research Council under the European Union's Seventh Framework Programme (FP7/2007-2013) / ERC grant agreement no.\ 335421. Jan Ob\l\'oj is also grateful to the Oxford-Man Institute of Quantitative Finance and St John's College in Oxford for their financial support. The authors thank Hans F\"ollmer and Vladimir Vovk for helpful comments and advice at various stages of this work.}
\maketitle
\section{Introduction}
 In a seminal paper, F\"ollmer \cite{fol81} pioneered a non-probabilistic approach to stochastic calculus. 
 For a function $x$ of real variable, he introduced a notion of quadratic variation  $\langle x\rangle_t$ along a sequence of partitions $(\pi_n)_n$ and proved the associated It\^o's formula for $f\in C^2$: 
\be f(x_t)-f(x_0)=\int_0^tf^\prime(x_s)dx_s+\frac{1}{2}\int_0^tf^{\prime\prime}(x_s)d\langle x\rangle_s , \ee
 where the integral $\int_0^tf^\prime(x_s)dx_s$ is defined as the limits of non-anticipative Riemann sums, shown to exist whenever $\langle x\rangle_t$ exists. F\"ollmer also observed that a path of a semimartingale 
a.s.\ admits quadratic variation in the pathwise sense and the usual stochastic integral agrees with his pathwise integral a.s.
 
The underlying motivation behind our current study was to extend the pathwise stochastic integral and its It\^o's formula to functions $f$ which are not in $C^2$. This question arose from applications in mathematical finance (see Davis et al.\ \cite{dor11}) but, we believe, is worth pursuing for its own sake. It led us to develop pathwise stochastic calculus featuring local times, which is the first main contribution of our work. We define a notion of local time $L^x_t(u)$ for a continuous function  $x$, prove the associated Tanaka-Meyer formula and show that a path of a continuous semimartingale $X$ a.s.\ admits pathwise local time $L^{X(\omega)}_t(u)$ which then agrees with the usual (stochastic) local time. 
Our contribution should be seen in the context of three previous connected works. First, our results are related to Bertoin \cite{ber87}, who showed similar results for a large class of Dirichlet processes; 
see also Coutin, Nualart and Tudor \cite{CoNuTu01} (who consider fractional Brownian motion with Hurst index $H>1/3$) and 
 Sottinen and Viitasaari \cite{SoVi14} (who consider a class of Gaussian processes).
 Second, related results appeared in the unpublished diploma thesis of Wuermli \cite{wur80}. Our approach is similar, however the proof in \cite{wur80} was complicated and applied only to square integrable martingales. We also have a slightly different definition of local time which includes continuity in time and importantly we consider convergence in $L^p$  for $p\in [1,\infty)$ instead of just $p=2$. This allows us to capture the tradeoff between the generality of paths considered and the scope of applicability of the Tanaka-Meyer formula. Indeed, as the term $\int_{\bR} L^{\pi}_t(u)f''(du)$ suggests, there is a natural duality between $L^{\pi}_t$ and $f''$, so the smaller the space to which $L^{\pi}_t$ belongs, the more general $f''$ one can take.
 This fact was already powerfully exploited by our third main reference, the recent paper by Perkowski and Pr\"omel \cite{PePr14} (and, to a much lesser extent, \cite{FeZh06}), in which $L^{\pi}_t$ belongs to the space of continuous functions, and thus $f''$ can be a general measure (i.e.\ $f'$ has bounded variation and $f$ is the difference of two arbitrary convex functions), recovering Tanaka-Meyer formula in full generality (the authors also consider the case where $L^{\pi}_t$ is continuous and also has bounded $p$-variation, and thus $f'$ can be any function with bounded $q$-variation). In particular, the conclusion of their main theorem (Theorem 3.5 in \cite{PePr14}) on the existence of $L^{X(\omega)}_t(u)$ for a.e. $\omega$ is stronger than ours; however, since their local time has to be continuous, results in \cite{PePr14} apply if $X$ is a local martingale either under the original probability $\bP$ or under some $\bQ \gg \bP$ (see  \cite[Remark 3.6]{PePr14}), whereas our Theorem \ref{GenExistsLT}  applies to a general semimartingale $X$. 
 
 Further, we investigate several  questions not considered in \cite{ber87}, \cite{wur80} and \cite{PePr14}, as we explain now. The main advantage of our definition, as compared with these previous works, is that  we are able to characterise the existence of pathwise local time with a number of equivalent conditions (see Theorem \ref{MainEquiv}). This feature seems to be entirely new. It allowed us in particular to build an explicit example of a path
 which admits a quadratic variation but no local time. Also,  while  \cite{dor11} and \cite{wur80} (following \cite{fol81})  consider partitions $\pi_n$ whose mesh is going to zero (which are well suited for changing  variables), the results \cite[Theorem 3.1]{ber87} and \cite[Theorem 3.5]{PePr14} of existence of pathwise local time consider Lebesgue\footnote{Which we define in \eqref{LebPart}.} partitions. Since neither type of partition is a special case of the other, this makes the results in these papers hard to compare. We solve this conundrum by proving our existence result (Theorem \ref{GenExistsLT}) for a general type of partition, which subsumes both types considered above.
Finally,  we prove  that the existence of $L_t(u)$ is preserved by a $C^1$ change of variables (improving on \cite[Proposition B.6]{dor11}) and by time changes, and that  
$g \to \int_0^t g(x_s)dx_s$ is continuous (similarly\footnote{Note that due to difference in definitions of local time, we use different topologies.} to  \cite{PePr14}).
 
Finally, we investigate how the limiting objects, quadratic variation and local time, depend on the choice of partitions. We show that for a path which oscillates enough, with a suitable choice of partitions, its quadratic variation can attain essentially any given non-decreasing function. From this, taking care of null sets and measurability issues, one can deduce that for a Brownian motion $W$ and a given increasing $[0,\infty]$-valued measurable  process $A$ with $A_0=0$   there exist refining partitions $(\pi_n)_n$ made of \emph{random}  times such that
\[ \textstyle  \langle W\rangle^{\pi_n}_t:= \sum_{  t_j \in \pi_n}(W_{t_{j+1}\wedge t}-W_{t_j \wedge t})^2 \to A_t\quad \textrm{a.s.\ for all }t\geq 0.\] 
This result illustrates, in the most stark way possible, the dependence of the pathwise quadratic variation (and thus of the pathwise local time) on the partitions $(\pi_n)_n$. This may push the reader to dismiss the notion of pathwise quadratic variation (and local time). However, it worth recalling that if we restrict ourselves to partitions constructed from  \emph{stopping} times, the limit of $ \langle W\rangle^{\pi_n}_t$ exists and is independent of the choice of partitions, and it always equals $t$. Analogously, our Theorem \ref{GenExistsLT} states that, if we only consider partitions constructed from stopping times, the pathwise local time $L^{X(\omega)}_t(u)$ of a semimartingale $X$ exists  and is independent of the choice of partitions, and it  coincides  with the classical local time. 

As already known by L\'evy,  $\inf_{\pi}\langle x\rangle^{\pi}_1=0$ for every continuous function $x$ and  $\sup_{\pi}\langle W(\omega)\rangle^{\pi}_1=\8$  for a.e.\ $\omega$. Our analysis builds on these  facts 
and answers in particular two questions which they leave open: whether for the general path $x=W(\omega)$ one can make $\langle x\rangle^{\pi_n}_1$ converge to any chosen  $C=C(\o)\in \bR$,  and what dependence in $t$ we can expect for $\lim_{n} \langle x\rangle^{\pi_n}_t$. Specifically it is clear that it must be an increasing function, and we wondered whether it is automatically continuous; indeed, while we followed F\"ollmer \cite{fol81}, who  carefully \emph{required} that  $\lim_{n} \langle x\rangle^{\pi_n}_t$ be continuous\footnote{More precisely \cite{fol81} deals with c{\`a}dl{\`a}g $x$ and requires that  $\mu_n$ (defined later in \eqref{MeasMu}) converge weakly to a measure $\mu$ which assigns mass $(\Delta x_t)^2$ to the singleton $\{t\}$; if $x$ is continuous this implies continuity of $ \langle x\rangle^{\Pi}_{\cdot}$.}, several authors who  cite \cite{fol81}  do not (see for example  \cite{ber87}, \cite{dor11}, \cite{son06}) and our results show that this is a significant omission.

 The plan for rest of the paper is as follows. In Section \ref{sec-pathwise} we introduce most of the notations and definitions, and recall parts of \cite{fol81}.
 In Section \ref{sec-plt} we identify several  conditions equivalent to the existence of pathwise local time, prove the Tanaka-Meyer formula and the continuity of $g \to \int_0^t g(x_s)dx_s$.
 In Section  \ref{sec-cv} we consider change of variable and time, and in Section \ref{sec-ext2conv} we extend Tanaka-Meyer formula from the case of a Sobolev function  $f$ to the case where $f$ is a difference of convex functions. 
  In Section \ref{sec-upXing} we prove that a path of a semimartingale a.s.\ admits the pathwise local time, and relate this to the downcrossing representation of semimartingale local time proved by L\'evy.
 Finally in Section \ref{sec-dep} we state the results about dependence of quadratic variation on the sequence of partitions, including the convergence $  \langle W\rangle^{\pi_n}_t \to A_t$ mentioned above. We only give the proof for one path avoiding the (non-trivial) technicalities related to measurability and null sets.
   The latter are given in the appendix.

\section{Pathwise stochastic calculus}\label{sec-pathwise}
In this section we introduce most notations and definitions used throughout the article, and we revisit the part of \cite{fol81} which deals with continuous functions, slightly refining its results  to include uniformity in $t$ and more general partitions.  

By  measure we mean sigma-additive positive  measure; a Radon measure will be the difference of two measures which are finite on compact sets.   With $|\mu|$ we will denote the total-variation measure relative to a `real measure' $\mu$ (i.e. $\mu$ is the difference of two measures), and with $\max(\mu,0)$ the  measure $(\mu+|\mu|)/2$ (i.e. the positive part in the Hahn-Jordan decomposition of $\mu$). 
We will say that $g_n\to g$ \emph{fast in} $L^p(\mu)$ if $\sum_{n\in \bN} || g_n -g||^p_{L^p(\mu)}<\8$; this trivially implies that $g_n\to g$ a.s. and in $L^p(\mu)$.
We will denote by $\cB$ (resp.\ $\cB_T$) the Borel sets of $[0,\infty)$ (resp.\ [0,T]). 
For a continuous function $x=(x_s)_{s\geq 0}$, $\udl{x}_t$ and $\ovl{x}_t$ are respectively the minimum and maximum of $x_s$ over $s\in[0,t]$. We set $x_{\8}=0$, denote by  $\delta_t$ the Dirac measure at $t$, and by  $\pi$  a partition of $[0,\8)$, i.e. $\pi=(t_k)_{k\in \bN}$ where $t_k\in [0,\8], t_0=0, t_k<t_{k+1}$ if $t_{k+1}<\8$, and $\lim_{k\to \8}t_{k}=\8$.
For such $x$ and $\pi$, we set 
\begin{align}
\label{Osc}
 O_t(x,\pi):=\max\{|x(b) - x(a)|:   a,b\in [t_k,t_{k+1})\cap [0,t] \textrm{ for some }k\in \bN \} .
\end{align} 
 F\"ollmer works with a sequence of finite partitions $(\pi_n)_n$ whose step on compacts converges to zero. This excludes very commonly used partitions: the \emph{Lebesgue partitions}, i.e., those of the form $\pi_P=\pi_P(x)=(t_k)_{k\in\bN}$,
\begin{align}
\label{LebPart}
 \text{ where }\quad 
t_0:=0  , \,  t_{k+1}:=\inf\{ t>t_k : x_t \in P, x_t\neq x_{t_k}  \} 
\end{align} 
 for some $P$  partition  of $\bR$, i.e., $P=(p_k)_{k\in \bZ}$  with  
 \[p_k\in [-\8,\8], \, \lim_{k\to \pm \8}p_{k}=\pm \8,\, p_0=0,  \text{ and } p_k<p_{k+1} \text{ if } p_k,p_{k+1} \in \bR .\]

We will work instead with partitions $\pi_n$ such that $O_T(x,\pi_n)\to 0$ for all $T<\8$; these are very flexible, as they subsume both Lebesgue partitions and the ones used by  F\"ollmer. Moreover they allow us to obtain time-change results, and have the additional advantage that one can always pass to refinements (since  if $\pi\subseteq \pi'$ then $O_T(x,\pi) \geq O_T(x,\pi')$).

While our aim in this paper is to develop a pathwise, non-probabilistic, theory, it is often the case that we want to consider paths that arise as sample functions of some stochastic process. Such processes are assumed to be defined on some underlying filtered probability space $(\Omega,\cF,(\cF_t)_{t\in [0,\infty)},\bP)$ satisfying the `usual conditions'. We denote by
 $\int_0^t H_s d X_s$  the stochastic integral of a predictable and locally-bounded integrand $H$ with respect to a continuous semimartingale $X=(X_t)_{t \geq 0}$.  Inequalities between random variables are tacitly supposed to hold for $\bP$-almost every $\omega$.
A sequence of partitions of $[0,\8)$ made of random (resp. stopping) times will be called a  \emph{random} (resp. \emph{optional}) \emph{partition} of $[0,\8)$; more precisely if $\pi=(\tau_k)_{k\in \bN}$, where $\tau_k$ are $[0,\8]$-valued random variables such that $ \tau_0=0, \tau_k\leq\tau_{k+1}$ with $\tau_k<\tau_{k+1}$ on $\{\tau_{k+1}<\8 \}$, and $\lim_{k\to \8}\tau_{k}=\8$, then $\pi$ is called a random partition, and if moreover $\{\tau_k\leq t \} \in \cF_t$ for all $k,t$ then $\pi$ is an optional\footnote{The terminology is justified by the fact that $\t$ is a stopping time iff $\1_{\{\tau_\leq \cdot \}}$ is an optional process.} partition.

 \begin{definition}\label{def-qvp}
 Given a continuous function $x:[0,\infty)\to\bR$ and a sequence of partitions $\Pi=(\pi_n)_n$ such that $O_T(x,\pi_n)$ converges to zero as $n\to \8$  for all $T<\8$, we will say that   \emph{$x$ has quadratic variation} (sampled along $\Pi$)     if  the measures 
 \begin{align}
\label{MeasMu}
\textstyle 
\mu_n:=\sum_{t_j\in\pi_n}(x_{t_{j+1}}-x_{t_j})^2\delta_{t_j}  
\end{align} 
converge vaguely\footnote{Meaning that $\int f d\mu_n \to \int f d \mu$ for every continuous function $f$ with compact support.} to a measure \emph{without atoms}  $\mu$ as $n\to\infty$. We will write $x\in \cQ$ if $O_T(x,\pi_n)\to 0$  for all $T<\8$ and $x$ has quadratic variation.
\end{definition} 
Recall that   $\mu_n$ converges weakly to a non-atomic measure $\mu$  
 iff its cumulative distribution function converges pointwise  to a continuous function, and thus iff
\[ \textstyle  \langle x\rangle^{\pi_n}_t:= \sum_{  t_j \in \pi_n}(x_{t_{j+1}\wedge t}-x_{t_j \wedge t})^2 \]
converges pointwise to a \emph{continuous} (increasing) function $\langle x\rangle_t$; the cumulative distribution function of $\mu$ is then $\langle x\rangle$, and is called the quadratic variation of $x$.
Such convergence is then always uniform in $t$, and more generally for every $T>0 $ and  continuous function $f: [\udl{x}_T,\ovl{x}_T]\to \bR$ 
\begin{align}
\label{QVInt}
\sum_{  t_j \in \pi_n}f(x_{t_i})(x_{t_{j+1}\wedge t}-x_{t_j \wedge t})^2 \longrightarrow  \int_0^t f(x_s) d \langle x\rangle_s  \text{ uniformly in } t\in [0,T] :
\end{align} 
 indeed if $t_j\leq t<t_{j+1}$ the sum on the left of \eqref{QVInt} differs from $\int_0^t f d\mu_n$ by at most 
 \[|f(x_{t_j})((x_{t_{j+1}}-x_{t_j})^2-(x_t-x_{t_j})^2)| \leq 2 ||f||_{\8} O_T(x,\pi_n)^2 , \]
  and  $\int_0^t f d\mu_n$ converges to $\int_0^t f d\mu$ uniformly in $t\leq T$ as the following simple observation applied to the positive and negative parts of $f$ shows.
\begin{scholium}
        \label{Polya}[Polya]
Let $F,F_n:[0,T] \to \bR$ be c{\`a}dl{\`a}g increasing. If $F$ is  continuous and $F_n \to F$ pointwise then the convergence is uniform in $t \in [0,T]$.
\end{scholium}

Note that a priori $\mu,\langle x\rangle$ and $\cQ$ depend on $\Pi=(\pi_n)_n$; when we want to stress this dependence, we will write $\mu^\Pi,\langle x\rangle^\Pi, \cQ^\Pi$. 
Note also that the series in \eqref{QVInt} is in fact a finite sum, since every partition is finite on compacts.

We introduce now some more notation which will be used throughout and in particular in Section \ref{sec-dep} and its proofs.  Given numbers $a\leq s\leq t\leq b$ and a finite partition $\pi$ of $[a,b]$ (meaning $\pi=(t_i)_{i=0}^k$ with $t_0=a$, $t_i<t_{i+1}$ for all $i$, and $t_k=b$) we set
\begin{align}
\label{QVst}
\langle x\rangle^{\pi}_{(s,t]}:=\sum_i (x_{(t_{k+1}\wedge t) \vee s}-x_{(t_k\wedge t ) \vee s})^2 ;
\end{align} 
 if  $s=a$ and $t=b$ the latter expression simplifies and we denote it with 
\begin{align*}
\label{}
\langle x\rangle_{\pi}:=\langle x\rangle^{\pi}_{(a,b]}=\sum_i (x_{t_{k+1}}-x_{t_k})^2 .
\end{align*} 
 Notice that  if $\pi$ is a partition of $[a,b]\ni s,t$ then
 \begin{align}
\label{TriangleIneqfor<B>}
\langle x\rangle^{\pi}_{(a,b]}=\langle x\rangle^{\pi}_{(a,s]}+\langle x\rangle^{\pi}_{(s,b]} \quad \text{ if } s\in \pi \, ;
\end{align} 
 in particular  
\begin{align}
\label{QVincreas}
\langle x\rangle^{\pi}_{(s,t]}   \leq  \langle x\rangle^{\pi}_{(a,b]} \quad \text{ if } s,t\in \pi ,
\end{align} 
and if  $\tilde{\pi}$ is a partition  of $[b,c]$ then
 \begin{align}
\label{QVadditive}
\langle x\rangle^{\pi\cup \tilde{\pi}}_{(a,c]}=\langle x\rangle^{\pi}_{(a,b]}+\langle x\rangle^{\tilde{\pi}}_{(b,c]}  .
\end{align} 

We shall now see that the quadratic variation sampled along optional partitions $(\pi_n)_n$ exists on a.e.\ path of a semimartingale and that a.e. it does not depend on $(\pi_n)_n$. This is essentially the usual result on the existence of the quadratic variation for a semimartingale.

\begin{proposition}
\label{ExistsQV}
Let $X$ be a continuous semimartingale and $[X]_t := X_t^2-2 \int_0^{t} X_s d X_s $. 
If   $\Pi=(\pi_n)_n$ are optional partitions such that $O_{T }(X,\pi_n)\to 0$ a.s. for all $T<\8$ then there exists some subsequence $(n_k)_k$ such that, for each $\omega$ outside a $\bP$-null set and setting $\Pi':=(\pi_{n_k})_k$, we have $X(\omega)\in \cQ^{\Pi'}$ and $\langle X(\omega)\rangle^{\Pi'}= [X](\omega)$.
\end{proposition} 
\begin{proof}
Write $\pi_{n}=(\tau_j^n)_j$, take   $H^n:=\sum_j X_{\tau_j^n}\1_{(\tau_j^n,\tau_{j+1}^n]}$ and notice that 
\begin{align}
\label{QVDiffInt}
 \textstyle 2\int_0^t XdX+ [X]_t= X^2_t=2\int_0^tH^n d X_t+ \langle X\rangle^{\pi_{n}}_t  \, .
\end{align} 
Since $H^n$ converges pointwise to $X$ and is bounded by a locally bounded predictable process\footnote{For example $|H_t^n|\leq X^*_{t}$ with $X^*_{t}:=\sup_{s\leq t}|X_s|$.}, the stochastic dominated convergence theorem gives that $\int_0^{\cdot} H^n d X$ converges to $\int_0^{\cdot} X d X$  uniformly on compacts in probability, which implies the thesis.
\end{proof} 
 
We now show that one can identify some of the subsubsequences along which the previous statement holds; in particular this holds when $\pi_n$ is the Lebesgue partition $\pi_{D_n}$ corresponding to $D_n:=2^{-n}\bN$ (the dyadics of order $n$).
Given $p\in [1,\8)$ we denote by $\cS^p$ the set of  continuous semimartingales   $X=(X_t)_{t\in [0,T]}$ which satisfy
\[ \|X\|_{\cS^p}:=\left\|[M]_{T}^{1/2} \right\|_{L^p} + \left\| \int_0^T\ d|V|_t \right\|_{L^p}<\infty ,\]
where $X=M+V$ is the canonical semimartingale decomposition of $X$, $[M]_t:=M_t^2-2 \int_0^{t} M_s d M_s $ is the quadratic variation of $M$ and  $|V|_t$ is the variation of $(V_s)_{s\in [0,t]}$. We recall the inequality 
\begin{align}
\label{GenBDG}
\| \sup_{t\leq T} \left| X_t \right| \|_{L^p(\bP)}  \leq C_p \|X\|_{\cS^p} \, ,
\end{align} 
which holds for local martingales (this being one side of the celebrated Burkholder-Davis-Gundy inequalities) and thus trivially extends to $X\in \cS^p$.
We will also use without further mention that if $H$ is  locally-bounded and predictable then the canonical decomposition of $\int_0^{\cdot} H dX$ is $\int_0^{\cdot} H dM+\int_0^{\cdot} H dV$ and so 
\begin{align*}
\label{}
\left\| \int_0^{\cdot} H dX \right\|_{\cS^p} =  \left\| \left(\int_0^T H_t^2  d[M]_t\right)^{1/2} \right\|_{L^p}  + \left\|\int_0^T |H_t| \ d|V|_t \right\|_{L^p}  .
\end{align*}

\begin{proposition}
\label{Qa.s.}
If in Proposition \ref{ExistsQV} we make the stronger assumption that $\sum_{n} O_{T }(X,\pi_n)<\8$ a.s. for all $T<\8$  then  $X(\omega)\in \cQ^{\Pi}$ and $\langle X(\omega)\rangle^{\Pi}=[X](\omega)$ for a.e. $\o$.
\label{ConvWholeSeq}
\begin{proof}
 Fix a compact time interval $[0,T]$ on which we will work. By prelocalizing  we can assume that $X\in \cS^4$ (see\footnote{This statement also appears in \cite[Chapter 5, Theorem 14]{Pr04}, without proof.} \'Emery \cite[Th\'eor\`eme 2]{Em79}) and passing to an equivalent probability we can  moreover\footnote{If $X\in \cS^4(\bP)$ and $(d\bQ/d\bP)(\o):=C \exp(- K(\o))$ then $K \in L^4(\bQ) $, and $X\in \cS^4(\bQ)$ since $d\bQ/d\bP\in L^{\infty}(\bP)$.} assume that  $K:=\sum_{n} O_{T }(X,\pi_n)\in L^4$.
Take $H^n$ as in Proposition \ref{ExistsQV}, $K^n:=H^n-X$ and notice that $\sup_t |K^n_t|\leq  O_{T }(X,\pi_n)$, so 
 $\sum_n \sup_t |K^n_t|\leq K$ and in particular $\sum_n \sup_t|K^n_t|^2\leq K^2\in L^2$.  Using   \eqref{GenBDG} and $(a+b)^2\leq 2(a^2 +b^2)$  gives that $\sup_t |\int_0^{t} K^n dX |$ converges to zero fast in $L^2$  if 
 $$ \textstyle \sum_n \bE ( \int (K^n)^2 d[M] + (\int K^n d|V|)^2 ) $$
 is finite, which is true since it is bounded above by 
 $ \textstyle \bE (  K^2 ([M]_T + |V|_T^2 )) $, which is finite by H\"older inequality  since $K\in L^4, X\in \cS^4$.
  Since $\sup_{t\leq T} |\int_0^{t} K^n dX |\to 0$ fast in $L^2$ and thus a.s.,  \eqref{QVDiffInt}  yields
  $\textstyle \sup_{t\leq T} |\langle X\rangle^{\pi_{n}}_t -  [X]_t| \to 0$ a.s.. 
\end{proof} 
\end{proposition}

\begin{theorem}[F\"ollmer \cite{fol81}]\label{fol} 
If $x\in\cQ$, $g\in C^1$ and $t\in [0,\infty)$ the limit
\begin{align}
\label{DefFolInt}
\lim_n \sum_{t_j\in \pi_n}g(x_{t_j})(x_{t_{j+1}\wedge t}-x_{t_{j}\wedge t})
\end{align} 
exists uniformly on compacts and defines a continuous function of $t$ denoted  $\int_0^tg(x_s)dx_s$. This integral satisfies It\^o's formula: for $f\in C^2(\bR)$,
\be f(x_t)-f(x_0)=\int_0^tf^\prime(x_s)dx_s+\frac{1}{2}\int_0^tf^{\prime\prime}(x_s)d\langle x\rangle_s.\label{ito}\ee
\end{theorem}
 Notice that the series defining the F\"ollmer integral in \eqref{DefFolInt} and later in this paper are in fact finite sums, since every partition is finite on compacts.

\begin{proof}  
By using the second order Taylor's expansion  write 
 \begin{align}
\label{TelSumf}
\sum_{t_j\in \pi_n}f(x_{t_{j+1}\wedge t})-f(x_{t_{j}\wedge t})
\end{align} 
as
\be
\sum_{t_j\in \pi_n} f^\prime(x_{t_j})(x_{t_{j+1}\wedge t}-x_{t_{j}\wedge t})+
\frac{1}{2}\sum_{t_j\in \pi_n} f^{\prime\prime}(x_{t_j})(x_{t_{j+1}\wedge t}-x_{t_{j}\wedge t})^2+ C_n(t)
\label{taylor} \, , \, \ee
 where the correction term $C_n(t)$ is bounded by
 \begin{align}
\label{SumPhi}
  \sum_{t_j\in \pi_n} \phi(|x_{t_{j+1}\wedge t}-x_{t_j\wedge t}|) (x_{t_{j+1}\wedge t}-x_{t_j \wedge t})^2 
\end{align} 
 for some increasing function $\phi$ which is  continuous at $0$ and such that $\phi(0)=0$. Since $x\in\cQ^{(\pi_n)_n}$, the  term \eqref{SumPhi} converges to $0$ (for $t=T $, and thus also uniformly in $t \leq T$). Since \eqref{QVInt} states that the second term of \eqref{taylor} converges to the last term of  \eqref{ito} uniformly on compacts, by difference the first term of \eqref{taylor} also converges,  uniformly on compacts; moreover \eqref{ito} holds since the telescopic sum \eqref{TelSumf} equals $f(x_t)-f(x_0)$. 
\end{proof} 

Remark that F\"ollmer \cite{fol81} considers sums of the form 
 \[ \sum_{\pi_n \ni t_j\leq t} g(x_{t_j})(x_{t_{j+1}}-x_{t_j}) \,  \text{  and   }  \, \sum_{\pi_n \ni t_j\leq t} g(x_{t_j})(x_{t_{j+1}}-x_{t_j})^2 ,\]
whereas we consider 
\begin{align}
\label{OurSums}
\sum_{t_j \in \pi_n}g(x_{t_j})(x_{t_{j+1}\wedge t}-x_{t_{j}\wedge t})  \,  \text{  and   }  \, \sum_{t_j \in \pi_n}g(x_{t_j})(x_{t_{j+1}\wedge t}-x_{t_{j}\wedge t})^2  . 
\end{align} 
Since the difference between these two sums is
 \[  g(x_{t_i}) (x_{t_{i+1}}-x_t)  \,  \text{  and   }  \, g(x_{t_i})((x_{t_{i+1}}-x_{t_i})^2 - (x_{t_{i+1}}-x_t)^2 )   \]
(where $i:=\max\{j: \pi_n \ni t_j\leq t\}$),  which goes to zero as $O_T(x,\pi_n)\to 0$, these expressions are equal in the limit. The reason we prefer to use \eqref{OurSums} is that it involves only non-anticipative quantities (i.e. their value of time $t$ does not depend on the value of $x$ at later times), which better fits with the theory of stochastic integration and thus allows us to obtain formulae like \eqref{QVDiffInt} and \eqref{LocTDifInt}.

\section{Pathwise local time} \label{sec-plt}

\bigskip\noindent As already suggested in \cite{fol81}, there should be an extension of `It\^o formula' valid also when $f''$ is not a continuous functions, as it is in \eqref{ito}. In the theory of continuous semimartingales, such an extension proceeds via local times and the Tanaka-Meyer formula; what follows is a pathwise version.

  \noindent If $f'_-$  is the left-derivative of a convex function $f$, and $f''$ is the second derivative of $f$ in the sense of distributions (i.e. the unique positive Radon measure which satisfies $f''([a,b))=f'_-(b)-f'_-(a)$)
we obtain for $a\leq b$
\bstar f(b)-f(a)&=& \int_a^b f\p_-(y)dy = \int_a^b\left(f\p_-(a)+\int_{[a,y)}f\pp(du)\right)dy \\
&=& f\p_-(a)(b-a)+\int_{[a,b)}(b-u)f\pp(du) \quad \textrm{, so that}\\
f(b)-f(a) &=& f\p_-(a)(b-a)+\int_{-\infty}^\infty\1_{[a\wedge b,a\vee b)}(u)|b-u|f\pp(du),\quad \forall a,b\in \bR.
\estar
So if given a function $x$  and a partition $\pi=(t_j)_j$,  
we set for $u,v \in \mathbb{R}$
\begin{align*}
\llbracket   u,v \rrparenthesis  := \begin{cases}
                                     [u,v), & \text{if } u \leq v,\\
                                     [v,u), & \text{if } u > v, 
                                   \end{cases}
\end{align*}
and define the discrete local time (along $\pi$) as
\be \label{Lpi}
\textstyle L^{\pi}_t(u):=2\sum_{t_j\in \pi} \1_{   \llbracket x_{{t_j}\wedge t}, x_{{t_{j+1}\wedge t}} \rrparenthesis}(u)|x_{t_{j+1}\wedge t} -u|,\ee
then, if $f$ equals the \emph{difference of} two convex functions, 
we have  the following discrete Tanaka-Meyer formula 
\be \label{nTanaka} 
f(x_t)-f(x_0) =\sum_{t_j\in \pi} f\p_-(x_{t_j})(x_{t_{j+1}\wedge t}-x_{t_j\wedge t}) + 
\frac{1}{2}\int_{\bR} L^{\pi}_t(u)f''(du) . \ee

 A simple but important remark is that only the values of $f$ in the compact interval $ [\udl{x}_t,\ovl{x}_t]$ are relevant.
Note that $L^{\pi}_t(u)=0$ for $u\notin [\udl{x}_t,\ovl{x}_t]$ and $L^{\pi}_t(\cdot)$ is c{\`a}dl{\`a}g, thus it is bounded; in particular $L^{\pi}_t(\cdot)$ is $f''$-integrable. 

In the remainder of this  section we will restrict our attention to those functions whose second derivative is not a general Radon measure but instead one which admits a density with respect to the Lebesgue measure. Thus, the underlying measure space will be $\bR$ with its Borel sets, endowed with the Lebesgue measure $\cL^1(du)$ (sometimes denoted simply by $du$).
We will consider $L^{\pi_n}_t(\cdot)$ as a function in $L^p$, and denote by
$W^{k,p}$ the (Sobolev) space of functions whose $k^{th}$ derivative in the sense of distributions is in $L^p$; i.e., $W^{1,p}$ is the set of absolutely continuous functions whose classical derivative (which exists a.e.) belongs to  $W^{0,p}=L^{p}$, and $W^{2,p}$ is the set of $C^1$ functions whose classical derivative belongs to $W^{1,p}$. 
The following is our main theorem in this section. 

\begin{theorem}
\label{MainEquiv}
Let $x$ be continuous function and fix a sequence of partitions $\Pi=(\pi_n)_n$ such that $O_t(x,\pi_n)\to 0$ as $n\to \8$ for all $t\in [0,\infty)$. Then, for  $1<p<\8$, $q=p/(p-1)$, the following are equivalent:
\begin{enumerate}
\item $\sum_{t_j\in\pi_n} g(x_{t_j})(x_{t_{j+1}\wedge t}-x_{t_j \wedge t}) $ converges
for every $g\in W^{1,q}$ and $t\in [0,\infty)$ to a continuous function of $t$, which we denote by $\int_0^t g(x_s) dx_s$.
\item  $\sum_{t_j\in\pi_n} g(x_{t_j})(x_{t_{j+1}\wedge \cdot}-x_{t_j \wedge \cdot}) $ converges uniformly on compacts for every $g\in W^{1,q}$.
\item \label{EqContMain} $(L^{\pi_n}_t)_n$ converges weakly in $L^p$ to some  $L_t$ for all $t\in [0,\infty)$, and  $[0,\infty) \ni  t\mapsto L_t\in L^{p}$ is continuous if $L^p$ is endowed with the weak topology.
\item For all $t\in [0,\infty)$ there exists $L_t\in L^p$ s.t. $ \int_{\bR} L^{\pi_n}_\cdot(u)h(u)du \to \int_{\bR} L_\cdot(u)h(u)du$  uniformly on compacts for every $h\in L^{q}$.
\item \label{BddQEq}  $x\in\cQ^{\Pi}$ and for all $M\in [0,\infty)$ there exists $t\geq M$ such that $(L^{\pi_n}_t)_n$ is bounded in $L^{p}$.
\end{enumerate} 
If the above conditions are satisfied then 
 $(L^{\pi_n}_t)_n$ is bounded in $L^{p}$ for all $t\in [0,\infty)$, and
 for all $f\in W^{2,q}$ and $t\in [0,\infty)$
\be  f(x_t)-f(x_0)=\int_0^tf^\prime(x_s)dx_s+\frac{1}{2}\int_{\bR} L_t(u)f''(u) du , \label{itoLT}\ee
and for all Borel bounded $h$
\be \int_0^t h(x_s)d\langle x\rangle_s= \int_{\bR} L_t(u) h(u)  du\,  .\label{od} \ee
The statements above  hold for $p=\8, q=1$ if the weak topology on $L^p$  is replaced by the weak$^*$ topology on $L^\8$.
Moreover, they also hold  for $p=1, q=\8$ if in item \ref{BddQEq} boundedness in $L^p$  is replaced by equintegrability.
 \end{theorem} 

In Theorem \ref{MainEquiv} we  slightly modify\footnote{Indeed \cite{wur80} does not require $t\mapsto L_t$ to be continuous, and considers strong convergence in $L^2$ instead of weak convergence in $L^p$.} the setting of \cite{wur80} in order to obtain a stronger theorem with equivalent conditions; the main novelty is that item \ref{BddQEq} implies the others. In particular we can exactly describe the difference between functions that only admit (pathwise) quadratic variation and the ones that also have  local time. In Example \ref{QVnoLT} below we show that the two notions are strictly different and give an explicit construction of a path which admits quadratic variation but not a pathwise local time.

We will henceforth denote by $\cL_p$ the space of continuous functions $x$ for which the equivalent conditions of Theorem \ref{MainEquiv}  hold. We will call $L_t(u)$ the \emph{pathwise local time of $x$} at time $t$ at level $u$. Observe that $\cL_p$ and $L_t(u)$ a priori depend on $\Pi=(\pi_n)_n$, and that $L_t(u)$ depends on $x$. We will  write $\cL_p^\Pi$, $L_t^{\Pi}(u)$, $L_t^x(u)$ or $L_t^{x,\Pi}(u)$ only when we want to highlight these dependencies; as in the remainder of this section  $\Pi=(\pi_n)_n$ will be fixed, we will never do that, and we  will simply write $L^n_t$  for $L^{\pi_n}_t$.

Notice that since $L^{n}_t(u)=0$ for $u\notin [\udl{x}_t,\ovl{x}_t]$, we can consider 
$L^{n}_t(u)$ as an element of $L^p(\mu)$ with $\mu$ being the restriction of the Lebesgue measure to 
$[\udl{x}_t,\ovl{x}_t]$. In particular,  Theorem \ref{MainEquiv} holds if $W^{k,q}$ is replaced with $W_{loc}^{k,q}$. Moreover, if $\hat{p}\leq p$, since $\mu$ is finite, $L^p(\mu)$ embeds continuously in $L^{\hat{p}}(\mu)$, and so  $\cL_p\subseteq \cL_{\hat{p}}$ and the  limits of $(L^{\pi_n}_t)_n$ in the weak $L^p$ and $L^{\hat{p}}$ topology coincide, so $L_t$ does not really depend on $p$.

Note also that for $x\in \cL_1$, using standard regularisation techniques, we can define a modification $(l_t)_t$ of the pathwise local time $(L_t)_t$ which is c\`adl\`ag and increasing in $t$ for a.e.\ $u$. The occupation time formula then extends to all Borel bounded $h$ 
\be \int_0^T h(t,x_t)d\langle x\rangle_t= \int_{\bR} \int_{0}^T  h(t,u) dl_t(u) du\, \label{od2}\ee

Finally, we show that if $x\in \cL_p$ the F\"ollmer integral is a continuous linear functional on $ W^{1,q}$. This fact could have been used to define F\"ollmer's integral for  $g\in W^{1,q}$ as the continuous extension of the  F\"ollmer's integral for  $g\in C^1$ defined in Theorem \ref{fol}, as done in \cite{ber87}. Note that the following result would not hold if we only assumed uniform convergence on compacts of $g_n$ to $g$.
\begin{proposition}
\label{Cont}
Let $p\in (1,\8]$ (resp. $p=1$) with conjugate exponent $q$. If $x\in \cL_p$, $g_n, g\in W^{1,q}$, $g_n(x_0)\to g(x_0)$ and  $g'_n\to g' $ in the weak (resp. weak$^*$) topology of $L^q$, then $\int_0^{t} g_n(x_s) dx_s \to \int_0^{t} g(x_s) dx_s$ for all $t\in [0,\infty)$, and  the convergence is  uniform on compacts if moreover
 $|g'_n|\to |g'| $ weakly (resp. weakly$^*$) in $L^q$.
\end{proposition} 
\begin{proof}
Define $f(u):=\int_{x_0}^u g(y) dy$ and analogously $f_n$ from $g_n$, and notice that $f_n(u)\to f(u)$ for all $u \in \bR$, so Tanaka-Meyer formula \eqref{itoLT} gives the thesis. If moreover  $|g'_n| \to |g'| $ weakly in $L^q$ then since the positive part $\max(h,0)$ of $h$ equals $(h+|h|)/2$, Polya's scholium \ref{Polya} shows local uniformity of the convergence $$\int_{\bR}  L_t(u) \max(g'_n(u),0) du  \to \int_{\bR}  L_t(u) \max(g'(u),0) du \, ;$$ working analogously with the negative parts we get the thesis.
\end{proof}

In the rest of this section we establish Theorem \ref{MainEquiv} via a series of lemmas; if not explicitly stated otherwise, $p$ is assumed to be in $(1,\infty)$.
\begin{lemma}\label{LsubQ}  $x\in\cQ$ iff $\int_{\bR} L^n_t(u) du$ converges to a continuous function $\psi_t$ of $t\in [0,\infty)$. In this case the convergence is uniform on compacts and $\langle x\rangle=\psi$.
 \end{lemma}

\begin{proof}
Applying formula  \eqref{nTanaka} with $f(x)=x^2\in L^{\1}([\udl{x}_t,\ovl{x}_t])$ 
 we obtain 
 \begin{align}
\label{QVn}
\textstyle \sum_{t_j\in \pi_n} x^2_{t_{j+1}\wedge t}-x^2_{t_{j}\wedge t}-2x_{t_{j}}(x_{t_{j+1}\wedge t}-x_{t_{j}\wedge t})=\int_{\bR} L^n_t(u)  du 
\end{align} 
The statement follows rewriting the left side of \eqref{QVn}  as $\sum_{t_j\in \pi_n}(x_{t_{j+1}\wedge t}-x_{t_{j}\wedge t})^2$.
\end{proof}

 Given $x\in \cQ$, $\nu_t$  will denote the \emph{occupation measure of $(x_s)_{s\leq t}$} (along $\Pi$), defined on the Borel sets of $[0,t]$  by $\nu_t(A):=\int_0^t 1_{A}(x_s)d\langle x\rangle_s.$

\begin{lemma}
\label{EquivCondLT}
If $x\in\cQ$ and $t\in [0,\infty)$ the following are equivalent.
\begin{enumerate}
\item \label{FInt8} For every $g\in W^{1,q}$ the following sequence converges
\begin{align} \textstyle
\label{FolIntW11} \sum_{t_j\in\pi_n} g(x_{t_j})(x_{t_{j+1}\wedge t}-x_{t_j \wedge t}) .
\end{align} 

\item \label{WeakConv} The sequence $(L^n_t(\cdot))_n$ converges in the weak  topology of $L^p$ (to a quantity which we denote by $L_t(\cdot)$).
\end{enumerate} 
 
The above conditions imply that $(L^n_t(\cdot))_n$ is bounded in $L^{p}$ and \eqref{itoLT} holds.
Conversely, if $(L^n_t(\cdot))_n$ is bounded in $L^{p}$ and $x\in\cQ$ then items \eqref{FInt8} and \eqref{WeakConv} hold, and  $\nu_t$ has a density $L_t$ with respect to $\cL^1$.

\end{lemma}

\begin{proof} 
The equivalence between items \eqref{FInt8} and \eqref{WeakConv}, and the fact that these imply \eqref{itoLT}, follows immediately applying \eqref{nTanaka} with $f(u):=\int_{x_*}^u g(y)dy$.
That item \eqref{WeakConv} implies the boundedness of $(L^n_t(\cdot))_n$ follows from Banach-Steinhaus Theorem. For the opposite implication notice that since $x\in\cQ$ we can use Theorem \ref{fol}, which together with \eqref{nTanaka}  shows that
\begin{align}
\label{Lambda} \textstyle
\exists \lim_n   \int_{\bR} L_t^n(u)h(u) du = \int_0^t h(x_s)d\langle x\rangle_s = \int h d\nu_t  \quad \text{ for all } h \in  C^0 \, .
\end{align} 

Since  $L^p$ is reflexive (see \cite[Theorem 4.10]{Bre10}), its unit ball is sequentially compact in the weak topology \cite[Theorem 3.18]{Bre10}, so we can get convergence of $L^n_t$ along some subsequence (of any subsequence) to some $L_t$ and all we have to show is that the limit does not depend on the subsequence.  Considering $(L^n_t)_n$ as elements of the measure space $([\udl{x}_t,\ovl{x}_t],\cL^1)$ we have that $C^0\subseteq L^p$, so $\int g(u) \nu_t (du) =\int g(u)L_t(u) du$ for all continuous $g$. Thus $L_t(u) du=\nu_t(du)$; in particular the limit $L_t$ does not depend on the subsequence, proving  item \eqref{WeakConv}.
\end{proof} 

\begin{lemma}
\label{contL}
If the equivalent conditions \ref{FInt8} and \ref{WeakConv} of Lemma \ref{EquivCondLT} are satisfied  for all $t\in [0,\infty)$, the following conditions are equivalent.
\begin{enumerate}
\item \label{FolCont} For every $g\in W^{1,q}$ the function $\int_0^t g(x_s) dx_s$ is continuous in $t\in [0,\infty)$.
\item \label{FolUnif}
For every $g\in W^{1,q}$ the convergence in \eqref{FolIntW11} is uniform on compacts.
\item \label{LTCont}  The map $[0,\infty) \ni  t\mapsto L_t(\cdot)\in L^{p}$ is continuous in the weak topology of $L^p$.
\item \label{LTUnif} For every $h\in L^{q}$ the convergence $ \int_{\bR} L^n_t(u)h(u)du \to \int_{\bR} L_t(u)h(u)du $ is uniform on compacts.
\end{enumerate} 
\end{lemma} 

 \begin{proof}
The identity \eqref{nTanaka} shows that items \eqref{FolUnif} and \eqref{LTUnif} are equivalent. The identity \eqref{itoLT} shows that \eqref{FolCont} and \eqref{LTCont} are equivalent. Trivially  item \eqref{FolUnif} implies item \eqref{FolCont}. Finally scholium \ref{Polya} shows that item \eqref{LTCont} implies item \eqref{LTUnif}.
\end{proof}

 \begin{proof}[Proof of Theorem \ref{MainEquiv}]
If item \ref{BddQEq} holds, since the last term in the decomposition \eqref{LpiUpDo} is bounded by $O_t(x,\pi)$ and the two sums are increasing in $t$, $(L^n_t)_n$ is bounded in $L^{p}$ for all $t\in  [0,\infty)$; moreover  Lemma \ref{EquivCondLT} shows that  for all  $h \in  L^q$
\begin{align}
\textstyle
\label{3eq}
\exists \lim_n   \int_{\bR} L_t^n(u)h(u) du = \int L_t(u) h(u)du=\int h d\nu_t  =\int_0^t h(x_s)d\langle x\rangle_s   \, ;
\end{align} 
Since \eqref{3eq} shows that $\int L_t(u) h(u)du $ is a continuous function of $t$,  Lemma \ref{contL} implies that item \ref{EqContMain} holds.
That item \ref{EqContMain} implies 
$x\in \cQ$ follows  applying Lemma \ref{LsubQ} since $\1_{[\udl{x}_t,\ovl{x}_t]}\in L^p$ and $L^n_t=0$ outside $[\udl{x}_t,\ovl{x}_t]$.
 Lemma \ref{EquivCondLT} states that $\nu_t$ has a density $L_t$; thus, formula \eqref{od} holds.
All other assertions follow directly from Lemmas \ref{EquivCondLT}  and \ref{contL}.

If $p=1$ or $p=\infty$ the proofs hold with the following minor modification in the part of the proof of Lemma \ref{EquivCondLT}  which deals with the sequential compactness of $(L^n_t)_n$.
If $p=\8$, the unit ball of $L^\8$ is sequentially compact
since it is compact  (and metrizable) in the
weak$^*$ topology because of Banach-Alaoglu Theorem (and since $L^1$ is separable), see \cite[Theorem 3.16]{Bre10} (and see \cite[Theorems 3.28 and 4.13]{Bre10}).
If $p=1$, since $L^n_t\geq 0$, Lemma \ref{LsubQ} implies that $(L^n_t)_n$ is bounded in $L^1$, so if $(L^n_t)_n$ is equintegrable then it is weakly sequentially compact (by the Dunford-Pettis Theorem, see \cite[Theorem 4.30]{Bre10}).
\end{proof} 
We end this section with the following 
\begin{example}
\label{QVnoLT}
There exists a function that admits pathwise quadratic variation but no pathwise local time.
\end{example} 
Put differently, we show that the inclusion $\cL_1^\Pi\subset \cQ^\Pi$ can be strict. This proves that the additional requirement in (5) in Theorem \ref{MainEquiv} is not automatic and is indeed needed. More precisely, we now construct a continuous function $x:[0,1]\to \bR$ and a sequence of refining partitions $\Pi=(\pi_n)_n$ of $[0,1]$ whose mesh is going to zero and such that 
$\langle x\rangle^{\pi_n}_{t}$ converges for all $t$ to the (continuous!) Cantor function $c(t)$ (a.k.a. the Devil staircase)  
but $(L_1^{\pi_n})_n$ does not converge weakly in $L^1(du)$; in particular $x$ has no pathwise local time along $(\pi_n)_n$, no matter which definition we use\footnote{Meaning that if one  replaced the weak topology of $L^1$ 
with any stronger topology (e.g. the weak/strong topology of $L^p$, or the uniform topology   as done in \cite[Definition 2.5]{PePr14}) one would still not obtain convergence of $L_t^{\pi_n}$.}. Our construction was inspired by a remark by Bertoin on page 194 of \cite{ber87}.

Divide $[0,1]$ into three equal subintervals and remove the middle one $I^1_1:=(\frac{1}{3},\frac{2}{3})$. Divide each of the two remaining closed intervals $[0,\frac{1}{3}]$ and $[\frac{2}{3},1]$ into three equal subintervals and remove the middle ones 
$I^2_1:=(\frac{1}{3^2},\frac{2}{3^2})$ and $I^2_2:=(\frac{7}{3^2},\frac{8}{3^2})$.
 Continuing in this fashion, at each step $i$ we remove the middle intervals 
$I^i_1, \ldots, I_{2^{i-1}}^i$, each of length $1/3^i$.
  The Cantor set is defined as $$C:= [0,1]\setminus \cup_{i=1}^{\infty} \cup_{j=1}^{2^{i-1}} I^i_j \, , $$
and the function which we will consider is $x(t):= \sqrt{2 \min_{s\in C} |s-t|}$. 
To construct our partitions $\pi_n$ of $[0,1]$ we define first a refining sequence 
$(\pi_{n,j}^i)_n$  of  Lebesgue partitions of  $I_{j}^i$ 
setting $\pi_{n,j}^i=(t^{k,i}_{n,j})_{k=0}^{2^{ni+1}}$ with $t^{0,i}_{n,j}=\inf I_j^i$ (so that $x(t^{0,i}_{n,j})=0$), and 
\begin{align} \textstyle
t^{k+1,i}_{n,j}:=\inf\{ t>  t^{k,i}_{n,j} :   x_t \in ( 2^{-ni} \sup_{t\in I_j^i} x(t))\bZ  ,\,   x_t \neq x_{t^{k,i}_{n,j}} \}    \, ,
\end{align} 
so that $t^{2^{ni+1},i}_{n,j}=\sup I_j^i $ and $|x_{t^{k+1,i}_{n,j}} - x_{t^{k,i}_{n,j}}|= 2^{-ni} \sup_{t\in I_j^i} x(t)=1/(\sqrt{3} 2^n)^{i}$  and so
\begin{align}
\label{QVNCantor} \textstyle
\langle x \rangle_{I^i_j}^{\pi_{n,j}^i}:= \sum_{k=0}^{2^{ni+1} -1}(x_{t^{k+1,i}_{n,j}} - x_{t^{k,i}_{n,j}})^2 = 2^{ni+1}  3^{-i} 2^{-2ni}=  3^{-i} 2^{1-ni}.
\end{align} 
Then define our refining sequence $(\pi_n)_n$ of partitions of $[0,1]$ whose mesh is going to zero setting  $\pi_n:=\{0,1\} \cup \cup_{i=1}^n \cup_{j=1}^{2^{i-1}} \pi_{n,j}^i $ and we set $\eps_n:=\frac{2}{2^n 3}$ so that as $n\to \infty$
\begin{align}
\label{QVCantor} \textstyle
\langle x \rangle_{1}^{\pi_{n}}= \sum_{i=1}^n \sum_{j=1}^{2^{i-1}}  \langle x \rangle_{I^i_j}^{\pi_{n,j}^i}   = 
 \sum_{i=1}^n (\eps_n)^{i}=  \frac{1- (\eps_n)^{n+1} }{1- \eps_n}\to 1 \, .
\end{align} 
Now, the Cantor function $c$ is defined on $[0,1]$ to be the only continuous  extension of the function $f$ which is defined on the set $D:=\{0,1\}\cup \cup_{i=1}^{\infty} \cup_{j=1}^{2^{i-1}} \bar{I}^i_j$ in this way\footnote{Such extension exists and is unique since $f$ is continuous and $D$ is dense in $[0,1]$.}: $f(0)=0,f(1)=1$, and each time we remove the middle third $I^i_j$ from a parent interval $J^i_j$, $f$ is defined on the closure $\bar{I}^i_j$ of $I^i_j$ to be the average of its values at the extremes of $J^i_j$ 
  (so $f=1/2$ on $\bar{I}^1_1$,  $f=1/4$ on $\bar{I}^2_1$ and $f=3/4$ on $\bar{I}^2_2$ etc.).

Since  the difference between $\langle x\rangle^{\pi_n}_{t} $ and the increasing function 
$ \sum_{\pi_n \ni t_j\leq t} (x_{t_{j+1}}-x_{t_j})^2$ is going to zero for all $t$  as $n\to \infty$, and since $c$ is continuous and increasing, to conclude
 that $\langle x\rangle^{\pi_n}_{t}\to c(t)$ for all $t$ it is enough to show it for all $t$ in the dense set $D$, see also Lemma \ref{ConvAtJumps} below. We already know this for $t=1$ and (trivially) for $t=0$. Since 
 $\langle x \rangle_{(0,\frac{1}{3}]}^{\pi_{n}}=\langle x \rangle_{(\frac{2}{3}, 1]}^{\pi_{n}}$, and  \eqref{QVNCantor} shows that
   $\langle x \rangle_{(\frac{1}{3},\frac{2}{3}]}^{\pi_{n}} \to 0$, \eqref{QVCantor} and \eqref{TriangleIneqfor<B>} give
 that  $\langle x \rangle_{t}^{\pi_{n}}\to 1/2=c(t)$ for all $t\in [\frac{1}{3},\frac{2}{3}]=\bar{I}^1_1$. Analogously 
 $\langle x \rangle_{(0,\frac{1}{9}]}^{\pi_{n}}=\langle x \rangle_{(\frac{2}{9}, \frac{1}{3}]}^{\pi_{n}}$, and  \eqref{QVNCantor} shows that
   $\langle x \rangle_{(\frac{1}{9},\frac{2}{9}]}^{\pi_{n}} \to 0$, so $\langle x \rangle_{\frac{1}{3}}^{\pi_{n}}\to 1/2$
    and \eqref{TriangleIneqfor<B>} give
 that  $\langle x \rangle_{t}^{\pi_{n}}\to 1/2^2=c(t)$ for all $t\in [\frac{1}{9},\frac{2}{9}]=\bar{I}^2_1$. In this way we see that  $\langle x\rangle^{\pi_n}_{t}\to c(t)$ for all $t\in D$ and thus for all $t\in [0,1]$.
 
To conclude, let us prove that the pathwise local time $L_1^{\pi_n}(u)$ converges to $0$ for all $u\neq 0$, so that $(L_1^{\pi_n})_n$ does not converge weakly in $L^1(du)$ (because otherwise, by the Dunford-Pettis Theorem \cite[Theorem 4.30]{Bre10}, it would be uniformly integrable and would thus converge \emph{to zero} strongly in $L^1(du)$, whereas we know that $\int_0^1 L_1^{\pi_n}(u) du=\langle x\rangle^{\pi_n}_{1}\to c(1)=1 $).
Since $x(t)\geq 0$ for all $t$, $L_1^{\pi_n}(u) =0$ for all $n$ if $u<0$.
For each $i,j$ the function $(x(t))_{t\in \pi_n \cap I^i_j}$ crosses\footnote{Meaning that either $x_{{t_k}}\leq u< x_{t_{k+1}} $ or $x_{t_{k+1}}\leq u< x_{t_{k}} $ where $(t_k)_k:=\pi_n \cap I^i_j$.} each level $u>0$ at most twice, and since $\sup \{ x(t): t\notin \cup_{i=1}^k  \cup_{j=1}^{2^{i-1}} I^i_j \}= 1/\sqrt{3^{k+1}}$ is strictly smaller than any $u>0$ for big enough $i=i(u)$, the number 
of times $(x(t))_{t\in \pi_n}$ crosses level $u>0$ 
is bounded above independently of $n$; since $O_1(x,\pi_n) = 1/(\sqrt{3}n)  \to 0 $ as $n\to \infty$, this implies that $L_1^{\pi_n}(u)\to 0$.

\section{Change of variables and time-change}\label{sec-cv}

 In applications to the study of variance derivatives, for example \cite{dor11}, one starts with a continuous positive price function $S$, and the `variance' is defined as the quadratic variation of the \emph{log price} $x=\log S$. In this connection it is useful to be able to change variables, and to relate for example the local time of $\log(x)$ with the one of $x$. We recall that, although being a semimartingale is preserved only by $C^2$ transformations, possessing a quadratic variation (in the sense of Definition \ref{def-qvp}) is more generally invariant under $C^1$ transformations; indeed  
 $f\in C^1$ and $x\in \cQ^\Pi$ imply $f(x)\in \cQ^\Pi$ and $\langle f(x)\rangle^\Pi_t = \int_0^t f^{'}(x_s)^2 \langle x\rangle^\Pi_s$ (see \cite[Proposition 2.2.10]{son06}).
We prove below a similar result for the pathwise local time (if $f$ is monotone), extending the $C^2$ case treated in \cite{dor11}; then we show that time-change preserves the pathwise local time.

For Propositions \ref{Sfx} and \ref{TimeChange} we consider a fixed  sequence of partition $(\pi_n)_n$ such that $O_t(x,\pi_n)\to 0$ as $n\to \8$ for all $t\in [0,\infty)$.

\begin{proposition}\label{Sfx} 
Let $x\in\cL_p$ and let $f:\bR\to\bR$ be $C^1$ and strictly monotone.
 Then $f(x)\in\cL_p$ and the pathwise local times of $x$ and $f(x)$ are related by
\be L^{f(x)}_t(f(u))=|f\p(u)|\,L^x_t(u).\label{Slt}\ee
\end{proposition}

In Proposition \ref{Sfx} one considers the same sequence of partitions $(\pi_n)_n$ for $x$ and for $f(x)$. This seems to be problematic, since ideally we would like  Proposition \ref{Sfx} to hold also for Lebesgue partitions, and clearly if $P$ is a partition of $\bR$ then $\pi_P(f(x))$ differs from $\pi_P(x)$. However Proposition \ref{Sfx} does apply to suitably chosen Lebesgue partitions since $\pi_{f(P)}(f(x))=\pi_P(x)$ if $f$ is strictly increasing.

To prove Proposition \ref{Sfx} and better understand the behavior of $L^{\pi}$, let  $t_J:=\max\{t_j\in \pi: t_j\leq t\}$  and
\begin{align}
\label{piUD}
\pi^U_t(u):=\{t_j \in \pi  : x_{{t_j}}\leq u< x_{t_{j+1}} , \, t_{j+1}\leq t\} \, , \\ 
\nonumber
\quad \pi^D_t(u):=\{t_j \in \pi : x_{t_{j+1}}\leq u< x_{t_{j}}   , \, t_{j+1}\leq t \} ,
\end{align} 
and notice that, since all the terms in \eqref{Lpi} with $ t<t_{j}$ are equal to zero, 
\be 
\textstyle \label{LpiUpDo}
L^{\pi}_t(u)/2=\sum_{t_j\in \pi^U_t(u)} (x_{t_{j+1}} -u) + \sum_{t_j\in \pi^D_t(u)} (u - x_{t_{j+1}} )  +  \1_{ [x_{t_J},x_t )}(u) |x_{t} -u| . 
\ee

\begin{proof}[Proof of Proposition \ref{Sfx}]
Since adding $t$ to any partition $\pi$ does not change the value of $L^{\pi}_t(u)$ and insures that the last term in  \eqref{LpiUpDo} is zero, we assume without loss of generality that our partitions contain $t$.
If  $f$ is strictly increasing  and $a\leq b$ then 
\[ x_{a}\leq u<x_{b}  \text{ iff } f(x_{a})\leq f(u)<f(x_{b}) ,\]
and thus \eqref{LpiUpDo} implies that $L^{f(x),\pi}_t(f(u))/2$ equals
\be 
\textstyle \label{Lpif}
\sum_{t_j\in \pi^U_t(u)} (f(x_{t_{j+1}}) -f(u)) + \sum_{t_j\in \pi^D_t(u)} (f(u) - f(x_{t_{j+1}}) )  .
\ee
If $t_j\in \pi^U_t(u)$, since $f\in C^1$ there exists $z_j(u)\in (u,x_{t_{j+1}})$  such that  
\[ f(x_{t_{j+1}}) -f(u)=f'(z_j(u))(x_{t_{j+1}} -u) , \]
so we can write the first sum in \eqref{Lpif} as
\begin{align}
\label{RemLtFU}
\sum_{t_j\in \pi^U_t(u)} (f'(z_j(u))-f'(u))  (x_{t_{j+1}} -u) +
f'(u) \sum_{t_j\in \pi^U_t(u)}  (x_{t_{j+1}} -u) .
\end{align} 
Treating analogously the second sum in \eqref{Lpif} we get that
\begin{align}
\label{diffLf}
L^{f(x),\pi}_t(f(u))- f'(u)L^{x,\pi}_t(u) 
\end{align} 
is bounded by 
\begin{align}
\label{RemLtFAll}
2 \sum_{t_j\in \pi^U_t(u) \cup \pi^D_t(u)} |f'(z_j(u))-f'(u)|  |x_{t_{j+1}} -u| .
\end{align} 
 Now define 
 \begin{align}
\label{}
 R_t(g,\pi):=\max\{|g(c) - g(d)|: c,d\in [\udl{x}_t,\ovl{x}_t],  |c-d| \leq O_t(x,\pi) \}  .
\end{align} 
Clearly \eqref{RemLtFAll} with $\pi=\pi_n$ is bounded by 
$ R_t(f',\pi_n)L^{x,\pi_n}_t(u)$, so since $L^{x,\pi_n}_t$ converges to $L^{x}_t$ and $ R_t(f',\pi_n)\to 0$  we get that \eqref{diffLf} with $\pi=\pi_n$ converges  to $0$, proving the thesis.

If $f$ is strictly decreasing then the argument is the same save for the sign change, which comes from the fact that upcrossings are now transformed in downcrossings and conversely, so $x_{t_{j+1}} -u$ needs to be replaced by 
$u-x_{t_{j+1}}$.
\end{proof}

\begin{proposition}\label{TimeChange} 
Let $\tau:[0,\8)\to [0,\8)$ be an increasing c{\`a}dl{\`a}g function such that $x_{\tau}$ is continuous and $\tau(0)=0$.
Given $\Pi=(\pi_n)_n$, let $\tau(\Pi):=(\tau_{\pi_n})_n$ where, given $\pi=(t_j)_j$, $\tau_\pi$ denotes the partition $(\tau_{t_j})_j$.
If $O_{\tau_t}(x,\tau_{\pi_n})\to 0$ for all $t\in [0,\infty)$ and $x\in\cL_p^{\tau(\Pi)}$  then 
 $O_t(x\circ\tau,\pi_n)\to 0$ for all $t\in [0,\infty)$, $x_\tau\in\cL_p^{\Pi}$ and 
 the pathwise local times are related by
\[ L^{x\circ\tau,\Pi}_t(u)=L^{x,\tau(\Pi)}_{\tau_t}(u). \]
Moreover if $\tau$ is bijective\footnote{I.e. if $\tau$ is strictly increasing, continuous and such that $\tau(0)=0, \lim_{t\to \8}\tau(t)=\8$.} then $x\circ\tau$ is continuous, $O_{\tau^{-1}_t}(x\circ \tau,\tau^{-1}_{\pi})=O_{t}(x,\pi)$ for any partition $\pi$, and if $P$ is a partition of $\bR$ then  the Lebesgue partitions of $x\circ \tau$ and $x$ satisfy $\pi_P(x\circ\tau)=\tau^{-1}_{\pi_P(x)}$.

\end{proposition}
\begin{proof} Even  if $\tau$ is not strictly increasing, the identity 
\[ \{\tau_s: s\in [t_i,t_{i+1})\cap[0,t] \}= [\tau_{t_i},\tau_{t_{i+1}})\cap[0,\tau_t]\cap \tau([0,\8)) ,   \]
holds, and it trivially implies that  $O_t(x\circ\tau,\pi_n)\leq O_{\tau_t}(x,\tau_{\pi_n})$, with equality if $\tau$ is a bijection.
Trivially $ L^{x\circ\tau,\pi}_t(u)=L^{x,\tau_\pi}_{\tau_t}(u) $
holds for every partition $\pi$, and everything else follows easily.
\end{proof}

Note that Propositions \ref{Sfx} and \ref{TimeChange} hold (with the same proof) with other definitions of existence of the pathwise local time; for example if one  replaced the weak topology of $L^p$ for $p\in[1,\8)$ (resp. the weak$^*$ topology on $L^\8$) with the strong one in item \ref{EqContMain} of Theorem \ref{MainEquiv}, or if one considered Definition 2.5 in \cite{PePr14}.

\section{Extension to convex functions }\label{sec-ext2conv}
 The choice of how to define the existence of the pathwise local time is intrinsically linked to the class of functions for which one is able to establish the pathwise Tanaka-Meyer formula \eqref{itoLT}.
To establish it for all convex functions one needs to restrict significantly the set of paths for which the local time exists; nonetheless, in \cite{PePr14} it is shown that this approach works for general enough paths (namely, for the `typical path' in the sense of Vovk \cite{vv12}).

It is natural to ask if the above can be extended even further, to all continuous functions. As already remarked  in \cite{fol81}, the next proposition shows that, if one wants to consider a generic path of a local martingale, the answer is no -- to define stochastic integrals in a pathwise manner for more general integrands one has to consider partitions which depend both on the integrator and the integrand as in \cite[Theorem 7.14]{Bi81} and \cite{Ka95}. 
 \begin{proposition} (Stricker \cite{str81}) Let $x\in C[0,T]$. If for every sequence of partitions $(\pi_n)_n$ with $O_T(x,\pi_n)\to 0$ and every  bounded continuous function $f$ on $\bR$ the Riemann sums $\sum_{t_i\in \pi_n}f(x_{t_i})(x_{t_i+1}-x_{t_i})$ converge, then $x$ has finite variation.
\end{proposition}

In what follows we take a different route from \cite{PePr14} to further extend F\"ollmer's integral and Tanaka-Meyer formula beyond $f\in C^2$. We consider $f$ which is a difference of two convex functions and write $f'_-$ for its left-continuous derivative and $f''$ for the second distributional derivative of $f$.
In a way somewhat reminiscent of \cite[Proposition 1.2]{ber87}, we define $ \int_0^t f'_{-}(x_s)dx_s$ as the limit of $\int_0^t f'_{n}(x_s)dx_s$, where $f_n$ are some special $C^2$ functions converging to $f$ and  $\int_0^{t} f'_n(x_s) d  x_s$ is defined in Theorem \ref{fol} as a limit of Riemann sums.  

We now fix $\Pi=(\pi_n)_n$ such that $O_t(x,\pi_n)\to 0$ as $n\to \8$ for all $t\in [0,\infty)$, and we consider a function $g$ which is $C^2$, positive and  with compact support in $[0,\8)$, and such that $\int_{\bR} g(x)dx =1$. We will then approximate the target fuction $f$ with $f_n:=g_n*f$, where  $*$ denotes the convolution between a function and a measure (or a function), $g_n$ is the mollifier $g_n(u):=ng(nu)$. 
Recall that, if $x\in \cL_1$,  $L_t(\cdot)$ is seen an element of $L^1(du)$; the following theorem assumes that there exists a modification of $L_t$ which is c{\`a}dl{\`a}g in $u$, i.e., a function $\tilde{L}_t(u)$ c{\`a}dl{\`a}g in $u$ and such that, for each $t$, the set $\{u:\tilde{L}_t(u)\neq L_t(u)\}$ has zero Lebesgue measure; this is not an unreasonable assumption, as it is satisfied by a.e. path of a semimartingale (indeed the local time of a continuous semimartingale has a modification which is jointly c{\`a}dl{\`a}g in $u$ and continuous in $t$).

 \begin{theorem}
\label{convLT}
  Assume that $x\in \cL_1$ and there exists a modification $L_t(u)$ of the pathwise local time  which is c{\`a}dl{\`a}g in  $u$  for all $t$.  If $f$ is convex then $f_n$ is $C^2$ and for all  $t\in [0,\infty)$ the  F\"ollmer integral    $\int_0^{t} f'_n(x_s) d  x_s$ converges  to a finite limit, denoted by $\int_0^t f'_{-}(x_s) d  x_s$,  which is independent of the choice of $g$ and satisfies 
\be 
 f(x_t)-f(x_0)=\int_0^t f'_{-}(x_s)dx_s+\frac{1}{2}\int_{\bR} L_t(u)f''(du)  
 \label{itoLTconv}. \ee
Moreover if $L_t(u)$ is jointly c{\`a}dl{\`a}g in  $u$ and continuous in $t$ then the convergence is  uniform on compacts and    $t\mapsto \int_0^t f'_{-}(x_s) d  x_s$ is continuous.
\end{theorem}

Theorem \ref{convLT} allows to define \footnote{As pointed out to us by F\"ollmer \cite{FolPriv}   another possible definition  of $\int_0^t f_{-}'(x_s) d x_s$ for non-smooth convex $f$ is as the 
 limit of $\int_0^t f'_k(x_s) d x_s$  for \emph{any} $(f_k)_k \subseteq C^2$ such that $f''_k(x)dx$  (considered as a measure) converges weakly to $f''(x)dx$. It follows from \eqref{itoLT} that  this definition makes sense (i.e. the limit exists and is independent of the approximating sequence $(f_k)_k$), and agrees with ours, if $L_t(u)$ has a modification  which is continuous in  $u$).}
  $\int_0^t g_{-}(x_s) d x_s$ for any function $g$ of finite variation on compacts, since then $f(u):=\int_{x_0}^u g(y) dy$ is the difference of convex functions.
 We now study the continuity properties of $g\mapsto \int_0^{\cdot} g_{-}(x_s) d x_s$. 
 \begin{proposition}
\label{ContConv}
Let $x\in \cL_1$ and assume that there exists a modification $L_t(u)$ of the pathwise local time  which is continuous in  $u$. If   $g_n,$ and $ g$ are functions of finite variation on compacts, $g_n(x_0)\to g(x_0)$ and  $g'_n\to g' $ weakly (seen as measures), then $\int_0^{t} g_n(x_s) dx_s \to \int_0^{t} g(x_s) dx_s$ for all $t\in [0,\infty)$. Moreover, if $|g'_n|\to |g'|$ weakly then the convergence is uniform on compacts.
\end{proposition} 
\begin{proof}
Define $f(u):=\int_{x_0}^u g(y) dy$ and analogously $f_n$ from $g_n$, and notice that $f_n(u)\to f(u)$ for all $u \in \bR$, so Tanaka-Meyer formula \eqref{itoLTconv} gives the thesis. If moreover  $|g'_n| \to |g'| $ weakly then, since the positive part $\max(h,0)$ of $h$ equals $(h+|h|)/2$, Polya's scholium \ref{Polya} shows local uniformity of the convergence $$\int_{\bR}  L_t(u) \max(g'_n,0)(du)  \to \int_{\bR}  L_t(u) \max(g',0) (du) \, ;$$ working analogously with the negative parts we get the thesis.
\end{proof} 

It is then natural to ask for which paths the above given definition of $\int_0^t f'_{-}(x_s) d x_s$ coincides with the one used  in Theorem \ref{MainEquiv} for $f\in W^{2,q}$. The answer is that  the limit of the Riemann sums 
 $\sum_{t_j\in\pi_n} f'_{-}(x_{t_j})(x_{t_{j+1}\wedge t}-x_{t_j \wedge t})$  exists and equals $\int_0^t f'_{-}(x_s) d x_s$ iff
$\int L_t^{x,\pi_n}(u)f''(du)  $ converges to $\int L^x_t(u)f''(du)  $, as it follows  from \eqref{nTanaka} and \eqref{itoLTconv}.
In particular this holds if $x\in \cL_p\subseteq \cL_1$, so the definition of the F\"ollmer's integral given in Theorem \ref{convLT} is indeed an extension of the one given in Theorem \ref{MainEquiv}.

 \begin{proof}[Proof of Theorem \ref{convLT}]
 Since $f$ is uniformly continuous on compacts, $f_n\to f$ pointwise. Thus, if we can prove that $\int_{\bR} L_{t}\,  df_n''\to \int_{\bR} L_{t} \, df''$, the thesis follows applying \eqref{itoLT} to $f_n$ and taking limits; indeed \eqref{itoLTconv} shows that $\int_0^t f'_{-}(x_s) d x_s$ does not depend on $g$. Define $\hat{g}_n(u):=g_n(-u)$ and apply Fubini's theorem and the identity $f_n''=g_n * f''$ to get that 
 \begin{align}
\label{switchLf}
 \int_{\bR} L_{t} \, df_n''= \int_{\bR} (\hat{g}_n * L_{t}) \,  df'' .
\end{align} 
 Since $L_t$ is zero outside $[\udl{x}_t,\ovl{x}_t]$ and $g$ has compact support,  $L_t, g$ and $\hat{g}_n * L_{t}$ are all $0$ outside a common compact interval $[-A,A]$.
In particular since   $L_t(\cdot)$ is c{\`a}dl{\`a}g it is bounded; since $\sup_u |\hat{g}_n * L_{t}(u)|\leq \sup_u |L_t(u)|$, the thesis follows from the dominated convergence theorem and \eqref{switchLf} if we prove that $\hat{g}_n * L_{t}(u) \to L_t(u)$ for all $u$.
 Notice that 
 \begin{align}
\label{*diff}
(\hat{g}_n * L_{t} - L_t)(u)=\int_{\bR} \hat{g}_n(y) (L_{t}(u-y)-L_{t}(u)) dy \, .
\end{align} 
 Since $L_t(\cdot)$ is right continuous,   for every $ \eps >0 $ and $u$ there is an $n$ such that 
 \begin{align}
\label{LTn}
|L_{t}(u-y)-L_{t}(u)|< \eps  \,  \text{ if   }  \, y\in [-A/n,0] ;
\end{align} 
 since $g=0$ outside $[0,A]$, the integral on the right side of \eqref{*diff} is actually over $[-A/n,0]$, so  $|\hat{g}_n * L_{t}(u) - L_t(u)| < \eps$. 
 
 Finally if $L_t(u)$ is jointly c{\`a}dl{\`a}g in  $u$ and continuous in $t$ then  $n$ such that \eqref{LTn} holds  can be chosen as to hold simultaneously for all $t$ in any given compact set. This  implies that   the convergence is  uniform on compacts,   and so $\int_0^{\cdot} f'_{-}(x_s) d  x_s$ is continuous.
  \end{proof}

\section{Upcrossing representations of local time}
\label{sec-upXing}
 In this section we will consider a  continuous semimartingale $X=(X_t)_{t}$ (with $t\in [0,\8) $ or $t\in [0,T]$) with  canonical semimartingale decomposition $X=M+V$ and with
  (classical\footnote{We refer to the semimartingale local time, i.e. the one for which the Tanaka-Meyer formula holds; this is in general different from the parallel notion of local time for Markov processes.}) local time $\bL_t(u)$ which is (jointly) continuous in $t$ and c{\`a}dl{\`a}g in $u$ (such a version exists, see \cite[Chapter 3, Theorem 7.1]{KarShr:91}). Some of our results specialise to the case where $\bL$ is jointly continuous in $t$ and $u$; this holds in the important case when $dV$ is absolutely continuous with respect to $d\langle M\rangle$ (this follows from \eqref{dLTpku} below, see also \cite[Example 2.2.3]{YenYor13}), in particular if $X$ is a local martingale (under the original probability $\bP$ or a  $\bQ$ such that $\bP\ll \bQ$).

The following is the  main theorem of this section. It essentially says that the pathwise local time sampled along optional partitions $(\pi_n)_n$ exists on a.e. path of a semimartingale, and that a.e. it equals the (classical) local time (in particular, it does not depend on $(\pi_{n})_n$).
 
\begin{theorem}
\label{GenExistsLT}
Assume that  $f:\bR\to \bR$ is the difference of two convex functions, that $\pi_n$ are optional partitions such that $O_{T }(X,\pi_n) \to 0$ a.s. and that $X=(X_t)_{t\in [0,\8)}$ is a continuous semimartingale. If $X$ has a  jointly continuous local time $\bL$, or if $f$ is $C^1$, then there exists a subsequence $(n_k)_k$ such that, for $\omega$ outside a $\bP$-null set (which may depend on $f''$), 
\begin{align}
\label{D-Lsub}
 \sup_{t\leq T} \left| L^{X(\o),\pi_{n_k}(\o)}_{t}(u) - \bL_{t}(\o,u) \right|  \to 0 \quad  \text{ in $L^p(|f''|(du))$ as $k\to \8$   }  \, 
\end{align} 
 simultaneously for all $p\in[1,\8)$, $T<\8$.
\end{theorem} 

 Note that applying Theorem \ref{GenExistsLT}  with $f(x)=x^2/2\in C^1$ gives in particular that a.e. path of a continuous semimartingale is in $\cL_p$ for all $p<\8$; indeed, 
$L^{X,\pi_{n_k}}_{t}(u) \to \bL_{t}(u) $ strongly (and thus weakly) in $L^p(du)$  $a.s.$, locally uniformly in $t$. 

The previous theorem follows from the following technical statement.

\begin{theorem}
\label{ExistsLT}
Let $\pi_n$ be optional partitions such that $O_{T }(X,\pi_n)\to 0$ a.s., $p\in[1,\8), T<\infty$, $X\in \cS^p$, $\mu$ be a sigma-finite positive Borel measure on $\bR$, and define 
 \begin{align}
\label{D-L}
h^{\pi_n}(u):= \left\| \sup_{t\leq T} \left| L^{X(\o),\pi_n(\o)}_{t}(u) - \bL_{t}(\o,u) \right| \right\|_{L^p(\bP(d\omega))}    \, , \, u\in \bR .
\end{align} 
Then $h^{\pi_n}(\cdot)$ is bounded and $(h^{\pi_n}(\cdot))_n$ converges pointwise (resp. $\mu$ a.e.) to $0$ if $\bL$ is jointly continuous (resp. if  $\mu$ is a measure  with no atoms).
 \end{theorem} 

The fact that $(h^{\pi_n}(\cdot))_n$ converges pointwise to zero was given an involved proof\footnote{The uniformity in $t$, not stated in \cite{wur80}, follows easily by Doob's $L^2$-inequality since \eqref{LocTDifInt} shows that $(L_t^{\pi_n,X}(u)-\bL_t(u))_t$ is a $L^2$ bounded martingale for each $u,n$.} in \cite{wur80} in the case where $X$ is a continuous martingale bounded in $L^2$, $p=2$ and $\pi_n$ are deterministic partitions such that $|| \pi_n||\to 0$. In the case where $X$ is in a class of continuous Dirichlet processes which includes $\cS^2$ semimartingales and the partitions are of Lebesgue-type, it is shown in \cite[Theorem 2.5 and Proposition 2.7]{ber87} that $L^{X(\o),\pi_n(\o)}_{t}(u) \to \bL_{t}(\o,u)$ weakly in $L^1(d\bP \times du)$ for  each $t$.

Moreover, Lemieux \cite[Theorem 2.4]{lem83} has derived 
a version of Theorem \ref{GenExistsLT} where the $L^p(|f''|(du))$ convergence is replaced by the uniform convergence, 
 in the special case where the partitions are of Lebesgue-type.
  For the case of continuous local martingales one can also  consult   Perkins \cite{Per79} or Chacon et al. \cite[Theorem 2 and Remark 2]{cjpt81}, or Perkowski and Pr\"omel \cite[Theorem 3.5 and Remark 3.6]{PePr14}, who actually prove convergence not only for $\bP$ a.e. $\omega$ but even quasi surely with respect to the set of all local martingale measures.
  Although our approach yields a weaker type of convergence, it has a simple proof and it works for continuous semimartingales and general optional partitions such that $O_{T }(X,\pi_n)\to 0$ a.s..

In the special case of Lebesgue partitions $\pi=\pi_{\epsilon \bZ}$,   Theorem \ref{ExistsLT}  closely relates to the downcrossing representation of local time conjectured by L\'evy (proved by It\^o and McKean for Brownian motion, extended by El Karoui to semimartingales, and found in \cite[Theorem VI.1.10]{ReYo99}), which  states that, for an $X\in \cS^p$,  
$$ \lim_{\eps \to 0} \left\| \sup_{t\leq T} \left| \eps D^{\eps}_t(\o,0)   - \bL_{t}(\o,0) \right| \right\|_{L^p(\bP(d\omega))} = 0  ,$$
where $D^{\eps}_t(\o,0) $ (defined in \eqref{LevelST} below) is the number of downcrossings at level $0$.
Indeed, as we now explain, L\'evy's representation above is equivalent to the fact that $h^{\pi_{\epsilon_n\bZ}}(0)\to 0$ whenever $0<\eps_n \to 0$. 

Given a continuous path $x=(x_s)_{s\leq t}$ and $a<b$, we set
 $\sigma^{a,b}_0:=0$, $\tau^{a,b}_0=\inf\{t:x_t=b\}$  and, for $k\geq 1$, we define
\begin{equation}
\label{LevelST}
\begin{split}
	&\sigma^{a,b}_k:=\inf\{t>\tau^{a,b}_{k-1}:x_t=a\}, \quad \tau^{a,b}_k:=\inf\{t>\sigma^{a,b}_k:x_t=b\},\\
& D^{\eps}_t(u):=\max\{k:\sigma_k^{u,u+\eps} \leq t\}.
\end{split}
\end{equation} 
It turns out that the downcrossings $D^{\eps}_t(u)$ of $(X_s)_{s\leq t}$ from $u+\eps$ to $u$ are closely related to the local time along $\pi_{\eps \bZ}$.  Indeed, the  upcrossings $U^{\eps}_t(u):=\max\{k:\tau_k^{u+\eps,u} \leq t\}$ of $(X_s)_{s\leq t}$ from $u$ to $u+\eps$ differ from $D^{\eps}_t(u)$ by at most $1$, so using   \eqref{LpiUpDo} we get that 
 \begin{align}
\label{LpiUpDo3}
L^{\pi_{\eps \bZ}}_t(u)/2=\textstyle U^{\pi_{\eps \bZ}}_t(u) (\eps - u ) + D^{\pi_{\eps \bZ}}_t(u) u +  \1_{ [x_{t_J},x_t )}(u) |x_{t} -u|  .
\end{align} 
The last term is bounded by $O_t(x,\pi_{\eps \bZ})\leq \eps$ and, considering $u=0$, we get that
\begin{align}
\label{LTvsDC}
\textstyle | L^{\pi_{\eps \bZ}}_t(0)/2 - \eps D^{\eps}_t(0) | \leq 2 \eps \, ,
\end{align} 
which concludes the proof of equivalence. 

We recall the following fact, for which we refer to \cite[Chapter 6, Theorem 1.7]{ReYo99}: 
\begin{align}
\label{dLTpku}
 2 \int_0^{\cdot}  \1_{\{X_s=u\}}   dX_s = 2 \int_0^{\cdot}  \1_{\{X_s=u\}}   dV_s =\bL_{\cdot}(u)-\bL_{\cdot}(u-) \quad \textrm{a.s.},\; \forall u\in \bR.
\end{align}

\begin{proof}[Proof of Theorem \ref{ExistsLT}]
Consider the convex function $f(x):=|x-u|$ and let $sign(x-u)$ be its left-derivative and  $2\delta_{u}$ its second (distributional) derivative. Subtracting from the discrete-time Tanaka-Meyer formula  \eqref{nTanaka}  its  continuous-time stochastic counterpart we get that
\begin{align}
\label{LocTDifInt}
0= \int_0^t  (H_s^{\pi_n}(u) -H_s(u)) dX_s+ ( L_t^{\pi_n,X}(u)-\bL_t(u))/2,
\end{align} 
  where using $\pi_n=(\t_i^n)_i$ we define the predictable processes 
\[H_s^{\pi_n}(u):=\sum_{i}sign(X_{\t_i^n}-u) 1_{(\t_i^n, \t_{i+1}^n] }(s) 
\,  \text{ and   }  \, 
H_s(u):=sign(X_s-u) .\]
Now  $h^{\pi_n}(u)\to 0$ follows from \eqref{GenBDG} and \eqref{LocTDifInt} if we show that 
$\int_0^{\cdot} H_s^{\pi_n}(u)  dX_s \to \int_0^{\cdot} H_s(u)  dX_s$ in $\cS^p$. To this end notice that 
\begin{align}
\label{domK}
 |H_s^{\pi_n}(u) -H_s(u) |\leq K_s^{\pi_n}(u):= 2 \times  1_{\{O_s(X,\pi_n)\geq |X_s-u| \}}  ,  
\end{align} 
and that since $X_{\cdot}$ and $ O_{\cdot}(X,\pi_n)$ are continuous adapted processes, $K_{\cdot}^{\pi_n}(u)$ is predictable, so  it is enough to prove that $ \int_0^{\cdot} K_s^{\pi_n}(u) dX_s \to 0  $ in $\cS^p$.
Since  $O_{T}^{\pi_n}\to 0$  a.s. implies that $K_{t}^{\pi_n}(u)\to 0$  a.s.  on $\{ X_t \neq u\}$ for all $t\leq T$, and since $K^{\pi_n}\leq 2$, the thesis follows from the (deterministic) dominated convergence theorem if $\| \int_0^{\cdot} 1_{\{ X_s= u\}} dX_s \|_{\cS^p}=0$, which by \eqref{dLTpku} holds for all $u$ if $\bL$ is continuous. Since Minkowski inequality for integrals says that  
 \[ \left\| \int_0^T   \1_{\{X_s=u\}}  d|V|_s  \right\|_{L^p(\mu)} \leq 
 \int_0^T  \| \1_{\{X_s=u\}} \|_{L^p(\mu)}  d|V|_s 
 = \int_0^T  ( \mu(\{X_s\}) )^{1/p}  d|V|_s \, , 
 \] 
 which is zero for $\mu$ which has no atoms. Using \eqref{dLTpku}, considering $L^p(\mu\otimes\bP)$ norm and using Fubini, we conclude that $ \| \int_0^{\cdot} 1_{\{ X_s= u\}} dX_s \|_{\cS^p}=0$ for $\mu$ a.e.\ $u$, and so  $h^{\pi_n}\to 0$ $\mu$ a.e.. Finally \eqref{GenBDG}, \eqref{LocTDifInt} and \eqref{domK}  imply that 
\[ h^{\pi_n}(u) \leq  C_p \left\|  \int_0^{\cdot}  2(H_s^{\pi}(u) -H_s(u)) dX_s  \right\|_{\cS^p}  \leq 4C_p \|  X \|_{\cS^p}  \quad  \text{ for all   }  u\in \bR  , \, \]
concluding the proof.
   \end{proof}

\begin{proof}[Proof of Theorem \ref{GenExistsLT}]
Let $(\tau_m)_m$ a sequence of stopping times which prelocalizes $X$ to $\cS^p$ (see Emery \cite[Theoreme 2]{Em79}), i.e. $\tau_m \uparrow \infty$ a.s. and  $X^{\tau_m -}\in \cS^p$ for all $m$.
  Let $\mu_i(A):=|f''|(A\cap [-i,i])$  and set
\[ \textstyle G_n(\o,T,u):=   \sup_{t\leq T} | L^{X(\o),\pi_{n}(\o)}_{t}(u) - \bL_{t}(u,\o) | \]
and $G_n^m:= \1_{\{T< \tau_m \}} G_n$.  Since $\mu_i$ is a finite measure, Theorem \ref{ExistsLT} implies that, as $n\to \infty$, $G_n^m$ converges  to  $0$ in $L^p(\bP\times \mu_i)$, for all $m,i\in \bN$ and $T\geq 0$. Passing to a subsequence (without relabelling) we can get convergence fast in $L^p(\bP\times \mu_i)$ and so, for $\omega$ outside a $\bP$-null set $N_{i,m}^{p,T}$,  $G_n^m(\o,T,\cdot)$ converges  to  $0$  in $L^p(\mu_i)$.
 Then along a diagonal subsequence  we obtain that $G_n^m(\o,T,\cdot)$ converges  to  $0$ in $L^p( \mu_i)$ for all $i,m,p,T \in \bN \setminus \{0\}$ for every $\o$ outside  the null set $N_{f''}:=\cup_{i,m,T,p\in \bN\setminus \{0\}} N_{i,m}^{p,T}$.
 Since $G_n=G_n^m$ on $\{T< \tau_m \}$, $G_n\to 0$ in $L^p( \mu_i)$ for all $i,p,T \in \bN \setminus \{0\}$ for every $\o$ outside  $N_{f''}$.
Since outside a compact set  $G_n(\o,T,\cdot)=0$  for all $n$,  convergence in $L^p(\mu_i)$ for arbitrarily big $i,p$ implies convergence in $L^p(|f''|)$ for all $p\in [1,\8)$. Since $G_n(\o,\cdot,u)=0$ is increasing, convergence for arbitrarily big $T$ implies convergence for all $T\in[0,\8)$.
\end{proof}

\section{Dependence on the partitions}\label{sec-dep}
In this section we investigate the extent to which the pathwise quadratic variation $ \langle x\rangle^{\Pi}:=\lim_n \langle x\rangle^{\pi_n}$ depends on the sequence of partitions $\Pi:=(\pi_n)_n$.
Instead of constructing explicit examples we  show that, for functions with a highly oscillatory behavior, the pathwise quadratic variation depends in the most extreme way possible on $(\pi_n)_n$. We then build on this fact and state how this applies to the general path of a Brownian motion; since taking care of 
all the thorny technicalities which arise from the dependence in $\o$ (i.e. tracking  the null sets and ensuring measurability) 
requires a long technical proof, we relegate this to the appendix.

Our work builds on two facts already mentioned (without proof) by L\'evy in \cite[Pag. 190]{le65}: that $\inf_{\pi}\langle x\rangle^{\pi}_1=0$ for \emph{every} continuous function $x$ and that for a.e. path $B(\omega)$ of a Brownian motion $\sup_{\pi}\langle B(\omega)\rangle^{\pi}_1=\8$. The corresponding proofs can be found in Freedman \cite[Pag. 47 and 48]{fre83}; the second fact can be found in a strengthened form and with an alternative proof in Taylor \cite[Corollary in Section 4]{tay72}. Our first result combines and generalises the above: we show that, with a suitable choice of $(\pi_n)_n$,  the pathwise quadratic variation may be equal to an arbitrary given increasing process $a$. Notice that we do not even make assume $a$ is right--continuous. We will denote with $D_n$ the dyadics of order $n$ in $[0,1]$, i.e $D_n:= [0,1] \cap \bN 2^{-n}$.

\begin{theorem}
	\label{ArbQvFn}
	Let $x:[0,1]\to \bR$ be a continuous function such that for every $0\leq c<d \leq 1$ there exist partitions $(\hat{\pi}_n)_n$ of $[c,d]$ such that 
	$\lim_n \langle x\rangle^{\hat{\pi}_n}_{(c,d]} = \infty$.
	Then, if $a:[0,1]\to [0,\infty)$ is an  increasing function such that $a_0=0$, there exist refining  partitions $(\pi_n)_n$ of $[0,1]$ such that $D_n\subseteq \pi_n$  for all $n$ and \begin{align} 
\label{<x>ConvA}
 \langle x \rangle^{\pi_n}_{t} \to a_{t}\quad   \text{ for all } t\in [0,1] \text{  as   }  \,  n\to \8 ,
\end{align} 
and the convergence is uniform if $(a_{t})_t$ is continuous.
 Moreover, given arbitrary partitions $(\bar{\pi}_n)_n$, one can choose the $(\pi_n)_n$  such that $\bar{\pi}_n  \subseteq \pi_n$ for all $n$.
 \end{theorem}
 To prove that convergence occurs at all times simultaneously, we will need the following simple lemma.
  \begin{lemma}
\label{ConvAtJumps}
Let $a:[0,1]\to [0,\infty)$ be increasing, $x:[0,1]\to \bR$ be continuous and  $(\pi_n)_n$ be partitions of $[0,1]$ such that $O(x,\pi_n)\to 0$, and assume that
\[  \langle x\rangle^{\pi_n}_t \to a(t)   \quad  \text{ for all   }\,  t\in F  \subseteq  [0,1]  \, \,  \text{ as   }  \,  n\to \8 . \] 
 If $F$ is dense in $[0,1]$ and contains the times of jump of $a$ then $\langle x\rangle^{\pi_n}_{\cdot} \to a_{\cdot}$ pointwise on $[0,1]$, and if $a$ is continuous  the convergence is uniform.
 \end{lemma} 
\begin{proof}
Although $\langle x\rangle^{\pi_n}$ is not necessarily an increasing function, it differs from the  increasing function $a^n(t):=\mu_n([0,t])$ (where $\mu_n$ is as in \eqref{MeasMu})  by at most $O(x,\pi_n)$, and so it is enough to prove the statement with $a^n$ replacing $\langle x\rangle^{\pi_n}$.
By hyphothesis  $\langle x\rangle^{\pi_n}_{t} \to a(t)$ for all $t$ at which $a$ is not continuous.
If $a$ is continuous at $t$ then for each $\e>0$ there exist $ s_1,s_2 \in F$ s.t. $s_1<t<s_2$ and $a(s_2)-a(s_1)<\e$, and so $a(t)-\e\leq a(s_1)=\lim_{n }\langle x\rangle^{\pi_n}_{s_1} \leq \liminf_{n } \langle x\rangle^{\pi_n}_{t}$ and analogously $\limsup_{n } \langle x\rangle^{\pi_n}_{t} \leq a(t)+\e$. Letting $\epsilon \downarrow 0$ we see that $ \lim_{n } \langle x\rangle^{\pi_n}_{t}$ exists and equals $a(t)$. Scholium \ref{Polya} concludes the proof.
\end{proof}

\begin{proof}[Proof of Theorem \ref{ArbQvFn}]
Note that, as observed already by  Freedman \cite{fre83},   given $k\in \bN\setminus \{0\}$ and $\pi$ we can build a partition $\pi' \supseteq \pi$ such that $\langle x\rangle^{\pi'}_t=\langle x\rangle^{\pi}_t/k$; indeed it is enough to do so on each subinterval of $\pi$, so we can assume that $\pi=\{c,d\}$. If $x(c)=x(d)$ take $\pi':=\{c,d\}$; if  $x(c)\neq x(d)$ we define $\pi'=(t_i)_{i=0}^k$ setting $t_0:=c$ and
 \[ t_i:=\min \{t\in [c,d]: x(t)=x(c)+(x(d)-x(c)) i/k \}  \quad  \text{  for  }  i= 1,\ldots, k \, ; \] 
 indeed $\langle x\rangle^{\pi'}_t=\sum_{i=0}^{k-1} ((x(d)-x(c))/k)^2 =\langle x\rangle^{\pi}_t/k$ holds. We will denote by $F(\pi,k)$ the partition $\pi'$ built with the above construction starting from $\pi$ and $k$.
 
We now fix $t$ and prove the existence of some $\pi'$ such that $| \langle x\rangle^{\pi'}_t -a(t)| \leq 1/2^n$; to do so we take $i\in \bN$ such that  $a(t)\in [i/2^n, (i+1)/2^n]$ and show that there exists $\pi$ such that $ \langle x\rangle^{\pi}_t \in [i,i+1]$ and then 
 take $\pi'=F(\pi,2^n)$; note that we automatically know such $\pi$ exists when $i=0$ (by taking $\pi=F(\tilde{\pi},k)$ where $\tilde{\pi}$ is an arbitrary partition and $k$ a big enough integer). If $i\geq 1$ since the quadratic variation over $[0,t]$ equals\footnote{This requires that $\pi$ contains each endpoint  of the subintervals; this does no harm, as it only means the $\pi$ we have to build must contain these points.} the sum of the quadratic variations over the subintervals $[mt/i,(m+1)t/i]$, $m=0,\ldots, i-1$, by time translation it is enough to prove that for any $s>0$ there exist $\tilde{\pi}'$ such that $\langle x\rangle^{\tilde{\pi}'}_s \subseteq [1,1+1/i]$.
As we assumed above,  there exist a partition $\hat{\pi}$ such that $\langle x\rangle^{\hat{\pi}}_s$ is arbitrarily large. Now using Freedman's construction  with  $k$ equal to be the integer part of $\langle x\rangle^{\hat{\pi}}_s$ we obtain $\tilde{\pi}'=F(\hat{\pi},k)$ such that $\langle x\rangle^{\pi'}_s\in [1,1+1/k] \subseteq  [1,1+1/i]$, concluding our proof of  the existence of some $\pi'$ such that $| \langle x\rangle^{\pi'}_t -a(t)| \leq 1/2^n$.  

Now, given any $\tilde{\pi}_n$, by applying the above reasoning to the increments of $a$ on each subinterval of $\tilde{\pi}_n$, we can find a $\pi_n\supseteq \tilde{\pi}_n$ such that $| \langle x\rangle^{\pi_n}_t -a(t)| \leq 1/2^n$ simultaneously for all $t \in \tilde{\pi}_n$.
If we  define $(\tilde{\pi}_n, \pi_n)$ by induction setting $\tilde{\pi}_0:=\{0,1\}=:\pi_0$ and taking  $\tilde{\pi}_n$ to be the union of  $D_n \cup \bar{\pi}_n$ with the times when $a$ has a jump of size bigger than $1/n$ and with $\cup_{k<n} \pi_k$, and then building $\pi_n$ from $\tilde{\pi}_n$ as explained at the beginning of this paragraph, we obtain a refining  $\pi_n$ which contains $D_n \cup \bar{\pi}_n$ and such that $\langle x\rangle^{\pi_n}_t \to a(t)$ holds  for any dyadic time and any time of jump of $a$, and thus holds simultaneously for all $t$ by Lemma \ref{ConvAtJumps}.

\end{proof} 

One can apply Theorem \ref{ArbQvFn} to the paths of Brownian motion; indeed, as L\'evy first remarked, for a.e. $\omega$ there exist partitions $ \pi_n=\pi_n(\omega)$ s.t. $\lim_n \langle B(\omega)\rangle^{\pi_n}_{(0,1]}=\8$. However, to obtain an interesting result one needs to show that the partitions can be chosen \emph{in a measurable way}. 
This requires first to correspondingly strengthen Levy's result in the following way.

\begin{lemma}
\label{StrongLevy}
If $0\leq c< d$ there exist \emph{random} partitions $\pi^n$ of $[c,d]$ such that 
\[  \langle B\rangle^{\pi^n}_{(c,d]} \to \8 \text{ a.s. as } n\to \infty .\]
\end{lemma}
 
 To prove Lemma \ref{StrongLevy}, one needs to revisit the proof of \cite[Theorem 1 and its Corollary in Section 4]{tay72} and delve into the proof of the existence of a Vitali subcover to show how one can choose a measurable one (on a set of probability arbitrarily close to $1$); although this essentially follows from an application of the section theorem, the proof is involved  
and we relegate it to the appendix.

Having established Lemma \ref{StrongLevy},  one can follow the logic of the proof of Theorem \ref{ArbQvFn} and with laborious but entirely elementary proofs\footnote{The proofs  rest entirely on Borel Cantelli's lemma and on the fact that, given a c{\`a}dl{\`a}g adapted process $D$, its jumps of size bigger than a given constant are stopping times, see \cite[Theorem 3.1]{Sio13DM}.} 
one can check measurability to obtain a similar result for the paths of Brownian motion, which we state below. 
To slightly generalize Theorem \ref{ArbQvFn} 
to include the case of a positive but potentially non-finite process $A$, we identify  $[0,1]$ (with the Euclidean topology) with $[0,\infty]$ using the  bijection $1-\exp(-x)$ (where $e^{-\8}:=0$), and thus  we endow  $[0,\8]$ with the distance $d(a,b):=|e^{-a} - e^{-b}|$,   which makes it homeomorphic to $[0,1]$  and for all $M<\8$ satisfies $d(a,b)\leq |a-b|\leq C d(a,b)$ for all $a,b\in [0, M]$ and some $C=C(M)$.
 Of course, if $A_{\cdot}(\o)$ is finite valued the convergence  under the Euclidean distance $|a-b|$ is equivalent to the  convergence under the distance $d(a,b)$.
In all that follows, if $\pi=(\tau_n)_n$ is a random partition we  denote by $\pi(\o)$ the sequence $(\t_n(\o))_{n\in \bN}$. 
Given sets $C,D$ which depend on $\omega$,
 we write that $C\subseteq D$ if $C(\omega) \subseteq D(\omega)$ for a.e. $\omega$, and in particular we say that a sequence of random partition $(\pi_n)_{n\in\bN}$ is \emph{refining} if $ \pi_n \subseteq \pi_{n+1}$ for all $n$.

 \begin{theorem}
\label{ArbQV}
Let $B$ be a Brownian motion and $A$ a jointly-measurable\footnote{Of course any \emph{c{\`a}dl{\`a}g} increasing process $A$ is jointly-measurable. However this is not true for general increasing processes: for example if $A:=Y \1_{\{\t\}} + \1_{(\t,\8)}$ where $\t$ is an exponentially distributed random time and $Y$ is a non-measurable function with values in $(0,1)$  then $A$ is a process with respect to the completed sigma-algebra ($A_t$ is measurable since $A_t=0$ a.e.), yet $A$ is clearly not jointly-measurable.}  increasing process with values in $[0,\8]$ and such that $A_0=0$. Then, there exist refining random partitions $\pi_n$ of $[0,\8)$ such that $ \bN 2^{-n} \subseteq \pi_n$ and  for all $ \o$ outside a null set
\begin{align} 
\label{<B>ConvA}
 \langle B(\omega) \rangle^{\pi_n(\omega)}_{t} \to A_{t}(\omega) \quad   \text{ for all } t\in [0,\8) \text{  as   }  \,  n\to \8 ;
\end{align} 
if $0\leq c<d<\infty$ the convergence in \eqref{<B>ConvA} is  uniform on $t\in [c,d]$ for the $\o$'s at which $(A_{t}(\o))_{t\in [c,d]}$ is continuous. Moreover, given arbitrary random partitions $(\bar{\pi}_n)_n$ one can choose $\pi_n=(\tau_n^i)_i$  such that $\bar{\pi}_n  \subseteq \pi_n$ for all $n$  and, if $A$ is adapted, $\tau_n^i+ 2^{-n}$ is a stopping time for each $i,n$.
\end{theorem} 
It is insightful to contrast this result with the well known fact that, if $(\pi_n)_n$ is a sequence of \emph{optional} partitions such that $O_{T }(B,\pi_n)\to 0$ a.e.,  $\langle B(\omega) \rangle^{\pi_n(\omega)}_t$ converges to $t$ uniformly on compacts in probability  (see the proof of Proposition \ref{ExistsQV}). The random times  $(\tau_n^i)_i$ making up $\pi_n$  do not need to look far into the future to break the convergence of $\langle B(\omega) \rangle^{\pi_n(\omega)}_t$ to $t$: as the theorem states,  one could take the $(\tau_n^i)_i$ to look only an arbitrarily small amount of time into the future.  Notice that the random times making up $\pi_n$ are bounded (since   $(\tau_n^i)_i=\pi_n \supseteq \bN 2^{-n}$ implies $\tau_n^i\leq i/2^n$).

Although we stated Theorem \ref{ArbQV} only for Brownian motion, it holds
for any continuous stochastic process with an oscillatory behavior wild enough to have infinite $2$-variation on any interval, in the sense that  for every $0\leq c<d<\infty$ there exist random partitions $\pi_n$ of $[c,d]$ along which the quadratic variation of $B$ converges a.s.\ to infinity.
 In particular our proof of Lemma \ref{StrongLevy} shows that Theorem  \ref{ArbQV} applies whenever $B$ is a continuous adapted\footnote{This is only used to obtain that $\tau_n^i+2^{-n}$ are stopping times; otherwise $B$ measurable is enough.} process for which
 there exists some continuous strictly increasing function $\psi$ such that $\psi(h)/h^2 \to 0$ as $h\downarrow 0$ and 
 \begin{align}
\label{BigIncr}
\text{ for  every $ t\geq 0$ }\quad   \limsup_{h\downarrow 0 } \frac{\psi(|B_{t+h} -B_{t}|)}{h} \geq 1  \quad  \text{ a.s. }  .
 \end{align}

 \appendix
\section{Proof of Theorem \ref{ArbQV}}
\label{appendixA}
 In order to deal with the technicalities involved in tracking the dependence in $\o$, we need to introduce a number of new definitions; these boil down to asking that, when evaluated at each $\o$, random partitions are (deterministic) partitions and the operations defined on them correspond to the analogous operations for (deterministic) partitions.
 Given two random times $\s\leq \t$, with slight abuse of notation we denote by $\st$  the set $ \{(\omega,t)\in \Omega\times [0,\8): \s(\o)\leq t \leq \t(\o)\} $. 
Given random times $ \s\leq \t$ we will  say that $\pi=(\t_k)_{k\in \bN}$ is \emph{a random partition of $\st$} if    $\tau_k$ are random times such that $ \tau_0=\s, \tau_k\leq\tau_{k+1}\leq \tau$ with $\tau_k<\tau_{k+1}$ on $\{\tau_{k+1}<\t\}$,    and  for a.e. $\o$ there exists some $k=k(\o)$ such that  $\tau_{k(\o)}(\o)=\t(\o)$; we then  denote by $K(\pi)$ the (\emph{finite}) random variable
\begin{align}
\label{Kpi}
K(\pi):=\min\{k\in \bN: \tau_{k}=\t \}   \,  \text{ if   }  \,  \pi=(\t_k)_{k\in \bN} .
\end{align} 
We denote by $\cP\st$ the set of random partitions of $\st$, with $\{\s\}$ the constant partition (i.e. $\{\s\}=(\s_i)_i$ with $\s_i=\s$ for all $i\in \bN$), and with 
$\cP[0,\8)$ the set of random partitions of $[0,\8)$ defined shortly before 
Definition \ref{def-qvp} (one could more generally define the random partitions of $[\s,\t)$; notice that for $(\tau_n)_n$ to be in $\cP[\s,\t)$ it is not required that $\tau_n=\t$ for big enough $n$, unlike $\cP[\s,\t]$, so the set $\cup_{n}\{\tau_n (\omega)\}$ does not need to be finite).
We now introduce several operations that one can perform on random partitions.
 Given random times $\alpha\leq \s\leq \t\leq \beta $ and  $(\t_k)_k=\pi\in \cP[\alpha, \beta]$ we define $\pi\cap \st$ to be the random partition 
 $((\t_k\wedge \t ) \vee \s)_k$ of $\st$; notice that $(\pi\cap \st)(\o)$ equals $(\pi(\o)\cap [\s(\o),\t(\o)])\cup \{ \s(\o), \t(\o)\}$. 
Given a measurable partition $(A_n)_{n\in \bN}$ of $\Omega$ and for each $n$ a random quantity $\t_n$ defined on $A_n$, one can  define on  $\Omega$  a random quantity  $\t$ by setting $\t:=\t_n$ on $A_n$; we will sometimes use this construction to define random times (and thus random partitions).
Given random partitions  $\pi=(\t_n)_n$  of $\st$ and $\tilde{\pi}=(\tilde{\t}_n)_n$ of $[\tilde{\s},\tilde{\t}]$, we define by induction the random 
 partition $\pi\cup \tilde{\pi}=(\rho_n)_n$ of $[\s\wedge \tilde{\s}, \t\vee \tilde{\t}]$ as follows: $\rho_0:=\s\wedge \tilde{\s}$, if $\rho_n(\o)=\t(\o)\vee \tilde{\t}(\o)$ then $\rho_{n+1}(\omega):=\t(\o)\vee \tilde{\t}(\o)$, and  if $\rho_n(\o)<\t(\o)\vee \tilde{\t}(\o)$ then 
 \begin{align*}
\label{}
\rho_{n+1}(\omega):=\min\{t > \rho_n(\omega) : t\in \pi(\omega) \cup \tilde{\pi}(\omega)\cup \{\t(\o)\vee \tilde{\t}(\o)\} \}   \, ;
\end{align*} 
  notice that the $\rho_n$ are indeed random times, as it follows from the following representation, where given a random time $\t$ and a measurable set $A$ we denote by $\t^A$ the random time $\t^A:=\t$ on $ A$ and $\t^A:=\8$ on $\Omega \setminus A$: \[ \rho_{n+1}=\min_k  (   \t_k^{\{\t_k >\rho_n \}} \wedge \tilde{\t}_k^{\{\tilde{\t}_k> \rho_n \}} \wedge (\t \vee \tilde{\t})) . \]
Given  $\pi^i \in \cP[\s_i,\t_i]$ for $i\leq k$  we can analogously  define $\cup_{i=0}^k \pi^i \in \cP[\min_i\s_i,\max_i \t_i]$; in particular when $\pi^i=\{ \s_i\}$ for each $i$ this defines 
$\cup_{i=0}^k \{ \s_i\}=(\rho_n)_{n\in \bN}$    as an element of $ \cP[\min_i\s^i,\max_i\s^i]$ (the point being that the $(\rho_n)_n$ are ordered whereas the $(\s_i)_i$ in general are not).
Notice that we cannot reasonably define the random partition $\cup_{i\in \bN} \pi^i$ for general $(\pi^i)_{i\in \bN}$;  indeed in general the set $\cup_{i\in \bN} \pi^i(\o)$ is not  finite, so there is no random partition $\pi'$ such that $\pi'(\o)=\cup_{i\in \bN} \pi^i(\o)$ for a.e. $\o$. 
However, if $\pi=(\t_i)_i \in \cP\st$ and $\pi^i \in \cP[\t_i,\t_{i+1}]$ for each $i\in \bN$, we can define $\cup_{i\in \bN} \pi^i=(\rho_n)_n\in \cP\st$ by induction like above: 
$\rho_0:=\s$, if $\rho_n(\o)=\t(\o)$ then $\rho_{n+1}(\omega):=\t(\o)$, and  if $\rho_n(\o)<\t(\o)$ then 
 \begin{align}
\label{CountUn}
\rho_{n+1}(\omega):=\min\{t > \rho_n(\omega) : t\in \cup_{i\in \bN} \pi^i(\o) \cup \{\t(\o)\} \}   \, ;
\end{align} 
 since for fixed $\o$ the set $\pi(\o)$ is finite, also $\cup_{i\in \bN} \pi^i(\o)$ is finite; thus the minimum in \eqref{CountUn} exists, and $\s_K=\t$ for some $K=K(\omega)$, so $\cup_{i\in \bN} \pi_i\in \cP\st$. Analogously  if $\pi=(\t_i)_i \in \cP[0,\8)$ and $\pi^i \in \cP[\t_i,\t_{i+1}]$ we can define $\cup_{i\in \bN} \pi^i=(\rho_n)_n\in \cP[0,\8)$ such that $\rho_n\leq \t_n$ for all $n$ by setting 
$\rho_0:=0$ and
 \begin{align*}
\label{}
\rho_{n+1}(\omega):=\min\{t > \rho_n(\omega) : t\in \cup_{i\in \bN} \pi^i(\o)  \}   \, ,
\end{align*}  
 where the minimum exists and $\rho_n\uparrow\8$ since $\cup_{i\in \bN} \pi^i(\o)$ is finite on compacts.

  Recall the definitions given in \eqref{QVst}--\eqref{QVadditive}. Note that given finite random times $\alpha\leq \s\leq \t\leq \beta $ and  $\pi\in \cP[\alpha, \beta]$  the random variable $\langle B\rangle^{\pi}_{(\s,\t]}$ is defined  by \eqref{QVst} path by path, i.e., $\langle B\rangle^{\pi}_{(\s,\t]}(\o):=\langle B(\o)\rangle^{\pi(\o)}_{(\s(\o),\t(\o)]}$. 
Thanks to the next simple lemma, in the rest of this section we only need to consider $[0,1]$-valued random times; in particular $\pi(\o)$ will be a finite set for a.e. $\o$ for any random partition $\pi$ and $\langle B\rangle^{\pi}_{(\s,\t]}$ will always be well  defined.

\begin{scholium}
\label{Time01}
It is enough to prove Theorem \ref{ArbQV} on  $\{t\in [0,1]\}$.
\end{scholium} 
\begin{proof}
For $k\in \bN$ applying Theorem \ref{ArbQV} on the time interval $[0,1]$ to the Brownian motion $B^k_t:=B_{t+k}-B_k$, the increasing process $A^k_t:=A_{t+k}-A_k$ and  $\bar{\pi}^k_n:=(\bar{\pi}_n\cap [k,k+1])-k$ produces a refining sequence $(\pi_n^k)_n\subseteq \cP[0,1]$ such that $\pi_n^k\supseteq \bar{\pi}^k_n $ for all $n$ and 
$\langle B \rangle^{\pi_n}_t$ converges a.s. to $A_t$   uniformly on $t\in [0,1]$.
 Since $(k)_{k\in \bN}$ is a `random' partition of $[0,\8)$ and $\pi_n^k +k \in \cP[k,k+1]$, we can define  $\pi_n:=\cup_{k\in \bN} \pi_n^k +k $, which is a random partition of $[0,\8)$ that trivially gives Theorem \ref{ArbQV}  on the time interval $[0,\8)$. 
\end{proof}

In the proof of the next lemma we will use the following notation: given a subset $E$ of $(0,1)\times \Omega$ and $\omega \in \Omega$, we set $E(\omega):=\{ t: (t,\omega)\in E\}$, $E(t):=\{ \omega: (t,\omega)\in E\}$ and $\Pi_{\Omega}(E):=\{\omega: (t,\omega)\in E \text{ for some } t\in (0,1) \}$.
With $\cB_1 \times \cF$ we will denote the product sigma algebra of the Borel sets  $\cB_1$ of $(0,1)$ with the underlying sigma algebra $\cF$ on $\Omega$; whenever a function of $(t,\omega)$ (or a subsets of $(0,1)\times \Omega$) is $\cB_1 \times \cF$-measurable, we will simply say that  is measurable.
We will assume that $\cF$ contains all null sets; this is without loss of generality because of \cite[Chapter 1, Lemma 1.19]{JacodShir:02}.
The Lebesgue measure on $(0,1)$ will be denoted with $\cL^1$, and the dyadics (resp. the dyadics of order $n$) in $(0,1)$ with $D$ (resp. $D_n$), i.e $D_n:= (0,1) \cap \bN 2^{-n}$ and $D=\cup_{n\geq 0} D_n$. As usual the $\inf$ (resp. $\sup$) of the empty set is defined to be $\infty$ (resp. $-\infty$).

\begin{proof}[Proof of Lemma \ref{StrongLevy}]  \textit{ Step 1.} 

Notice that Brownian motion satisfies  \eqref{BigIncr} with $\psi(h):=h^2/(2\log \log (1/h))$ since  $\phi(h):=\sqrt{2h \log \log (1/h)}$ is asymptotically inverse to $\psi$ (meaning that $\phi(\psi(h))/ h$ $ \to 1$ and $\psi(\phi(h))/ h\to 1$ as $h\downarrow 0$)
and by the the iterated log law 
$$ \limsup_{h\downarrow 0 } \frac{|B_{t+h} -B_{t}|}{\phi(h)} = 1   \text{ a.s. }. $$
Moreover also the process $X_t:=(B_{t(b-a)+a}-B_{a})/\sqrt{b-a}$ satisfies  \eqref{BigIncr} (with the same $\psi$), so  we can without loss of generality take $a=0, b=1$.
For $k,n\in \bN$, $t,\epsilon \in (0,1)$ and $h\in F\subset (0,1]$ define $ Y^h_t:=\psi(|B_{t+h} -B_t|) / h$,
\begin{align}
\label{E(F)}
E(F):=\{ (t,\omega)\in (0,1)\times \Omega : Y^h_t >1- \epsilon  \text{ for some } h\in F \} \, ,
\end{align} 
$E^n_k:=E((0,1/2^k] \cap D_n)$ and $E_k:=E((0,1/2^k])$. Since  $Y_t^h$ and $Y_t(n):=\max \{Y^h_t: h\in (0,1/2^k] \cap D_n\}$  are continuous in $t$ and measurable in $\omega$,   $Y_{\cdot}^h$ and $Y_{\cdot}(n)$   are measurable.
It follows that $E^n_k=\{ Y_{\cdot}(n)>1-\epsilon \}$ is measurable , and so also is $\cup_n E_k^n=E((0,1/2^k] \cap D)$. Since $Y^h_t$ is continuous in $h$, $E_k$ equals $E((0,1/2^k] \cap D)$ and thus it is measurable, and in particular $E:=\cap_{k} E_k$ is measurable. 
Notice that \eqref{BigIncr} shows that $\bP(E_k(t)) =1$ for each $t,k$, and so 
$\bP(E(t)) =1$ for each $t$.
 Fubini's theorem applied to the product of $\bP$ with $\cL^1$ shows that  $\cL^1(E_k(\omega)) =1=\cL^1(E(\omega)) $ for $\bP$ a.e. $\omega$, and in particular $\Pi_{\Omega}(E)=\{\omega: E(\omega)\neq \emptyset \}$ has probability $1$. 
Define for every $\omega\in \Pi_{\Omega}(E)$ and $n\in \bN$
 $$\mathcal{J}^n(\omega):=\{[t,t+h]: h\in (0,1/2^n], h+t<1, (t,\omega)\in E \text{ and } Y^h_t(\omega) > 1-\epsilon \} ; $$ 
since by definition $E$ is the set of $(t,\omega)$ for which there exist arbitrarily small $h>0$ such that $Y^h_t>1-\epsilon$, $\mathcal{J}^n(\omega) $  is a Vitali cover of $E(\omega)$ for every $\omega\in \Pi_{\Omega}(E)$. It follows from Vitali's covering theorem \cite[Theorem 1.31]{Le09} that for every $\omega\in \Pi_{\Omega}(E)$ there exist $N^n(\omega)<\infty$ and  $((\mathfrak{t}_i^n ,\mathfrak{h}_i^n)(\omega))_{i=1}^{N^n(\omega)}$ such that for all $i\neq j$ and $\omega\in \Omega$
\begin{align}
\label{thN} \textstyle
 \mathcal{J}^n(\omega) \supseteq [\mathfrak{t}_i^n ,\mathfrak{t}_i^n +\mathfrak{h}_i^n](\omega)=:I_i^n(\omega)
\, , \, 
I_i^n(\omega) \cap I_j^n(\omega) =\emptyset
\, , \, 
\sum_{i=1}^{N^n} \mathfrak{h}_i^n(\omega) > 1-\epsilon \, .
\end{align} 
Assume for the moment that $\mathfrak{t}_i^n,\mathfrak{h}_i^n,N^n$ depended measurably in $\omega$. Since on $\Pi_{\Omega}(E)$
\begin{align*}
Y^n:= \sum_{i} \psi(|B_{\mathfrak{t}_i^n+\mathfrak{h}_i^n} -B_{\mathfrak{t}_i^n}|) \geq  \sum_{i} \mathfrak{h}_i^n (1-\epsilon)  \geq   (1-\epsilon)^2 >0 \, 
\end{align*} 
and any interval in $\mathcal{J}^n(\omega)$ has length at most $1/2^n$,  $\lim_{s\downarrow 0}\psi(s)/s^2= 0 $  implies that
$$  \textstyle   \sum_{i} (B_{\mathfrak{t}_i^n+\mathfrak{h}_i^n} -B_{\mathfrak{t}_i^n})^2   \to \infty  \quad \text{ on }\Pi_{\Omega}(E) \text{ and so a.s.,}  $$
and so if $\pi^n=(s_k)_k$ is the random partition made of $0,1$ and the points $\mathfrak{t}_i^n$ and $\mathfrak{t}_i^n+\mathfrak{h}_i^n$ for $i=1,\ldots, N^n$ we get that $\langle B\rangle^{\pi^n}_{(0,1]}\to \infty$ a.s..

\textit{ Step 2.} Thus, to conclude the proof it is enough to show that one can choose $\mathfrak{t}_i^n, \mathfrak{h}_i^n,N^n$ which depend measurably in $\omega$ and satisfy \eqref{thN} \emph{for all $\omega$}. While we cannot quite do that, by revisiting the proof of Vitali's covering theorem and applying the section theorem (for an elementary proof of which we refer to \cite[Theorem 3.1]{Bass10}, \cite{Bass11}) and its immediate corollary \cite[Chapter 1, Theorem 4.14]{ReYo99} we obtain  measurable $\mathfrak{t}_i^n, \mathfrak{h}_i^n,N^n$ which  satisfy \eqref{thN} for all $\omega\in V_n$, where $V_n$ is a large set, and this allows us to conclude the proof as we explain after \eqref{thprops}. While on $V_n$ necessarily $\mathfrak{t}_i^n, \mathfrak{h}_i^n \in (0,1]$ for $i\leq N^n$, in general our $\mathfrak{t}_i^n, \mathfrak{h}_i^n $ may also take the values $0$ and $\infty$. 
Indeed, in Steps 3,4 we will construct  by induction on $i\geq 1$ random times
 $\mathfrak{t}_i^n, \mathfrak{h}_i^n$ which satisfy 
the properties stated below in \eqref{thprops} relative to the objects which we will now define. Set
$$ V_n:=  \Pi_{\Omega}(E) \cap (\cap_{i \geq 1} \{\mathfrak{t}_i^n +  \mathfrak{h}_i^n<\infty\})$$  and define the decreasing family of random intervals $(\mathcal{J}^n_i)_i$ setting $\mathcal{J}^n_1:=\mathcal{J}^n\neq  \emptyset$ on $\Pi_{\Omega}(E)$ and $\mathcal{J}^n_1:= \emptyset$ otherwise and, given 
\begin{align}
\label{setC}
 C_i^n:= \{\mathfrak{t}_i^n +  \mathfrak{h}_i^n<\infty\} \cap \{\mathcal{J}^n_i\neq \emptyset \} \, ,
\end{align} 
 we define by induction 
\begin{align*}
  \mathcal{J}^n_{i+1}:=\{[t,t+h]\in \mathcal{J}^n_i : [t,t+h] \cap 
[\mathfrak{t}_i^n, \mathfrak{t}_i^n + \mathfrak{h}_i^n]  = \emptyset \} \text{ on } C_i^n \, ,  \quad   \mathcal{J}^n_{i+1}:=\emptyset \text{ otherwise} \, ,
\end{align*}
and then we set
\begin{align*}
S^i:=\sup\{   h\in (0,1] : [t,t+h] \in \mathcal{J}^n_i, t\in (0,1) \} \vee 0 \, ,
\end{align*} 
so that $\{S^i=0\}= \{\mathcal{J}^n_i= \emptyset \}$. 
In Steps 3,4 we will construct $\mathfrak{t}_i^n, \mathfrak{h}_i^n$  such that 
\vspace{-0.1cm}  
\begin{equation}
 \begin{array}{c} \label{thprops}  
[\mathfrak{t}_i^n, \mathfrak{t}_i^n + \mathfrak{h}_i^n] \in \mathcal{J}^n_i \, \text{ and} 
\quad 2\mathfrak{h}_i^n>S^i>0  \text{ on } C_i^n  \vspace{0.3cm}     \\ 
 \mathfrak{h}_i^n=0  \text{ on }
  \{\mathcal{J}^n_i= \emptyset \} =\{ \mathfrak{t}_i^n =0\} , \quad \bP(V_n)\geq 1-1/2^n  ;
\end{array}
\end{equation}
once obtained such random times $(\mathfrak{t}_i^n ,\mathfrak{h}_i^n)_{i\geq 0}$, the proof proceeds as follows. 
From \eqref{thprops}, the proof of Vitali's covering theorem (see  \cite[Theorem 1.31]{Le09}) shows that
 $\cL^1(E(\omega)\setminus  \cup_{i\geq 1} [\mathfrak{t}_i^n, \mathfrak{t}_i^n + \mathfrak{h}_i^n](\omega)  ) = 0$ for $\omega\in V_n$, and so
 (since we proved that $\cL^1(E(\omega))=1$) $N^n:=\inf\{k: \sum_{i=1}^k \mathfrak{h}_i^n >1-\epsilon \}$ is a \emph{finite} random variable on $V_n$. Notice that $\mathfrak{h}_i^n>0$ for all $i\leq N^n$, since 
 $ \{\mathfrak{h}_i^n=0\}=\{\mathcal{J}^n_i= \emptyset \}$ is increasing in $i$.
 Thus $N^n$ and  $(\mathfrak{t}_i^n ,\mathfrak{h}_i^n)_{i=1}^{N^n}$ satisfy \eqref{thN} on $V_n$ for all  $i\neq j$, and thus if $\pi^n=(s_k)_k$ is the random partition made of $0,1$ and the points $\mathfrak{t}_i^n$ and $\mathfrak{t}_i^n+\mathfrak{h}_i^n$ for $i=1,\ldots, N^n$, reasoning as in Step 1 we have that $Y^n \geq   (1-\epsilon)^2$ on $V_n$. It follows that $   \langle B\rangle^{\pi^n}_{(0,1]}(\omega) \to \8 $  if $\omega\in V_n$ for infinitely many $n$'s, and so by Borel Cantelli's lemma  $   \langle B\rangle^{\pi^n}_{(0,1]}(\omega) \to \8 $ a.s. since $\bP(\Omega \setminus V_n)\leq 1/2^n$.

\textit{ Step 3.} 
To conclude the proof we need, for fixed $n$, to define  random times $\mathfrak{t}_i^n, \mathfrak{h}_i^n$ which satisfy \eqref{thprops}. We will so do in Step 4, using the auxiliary processes $L^i$ which we introduce in this step.  Given $\mathcal{J}^n_i$ (which so far we only defined for $i=1$) and $F\subseteq (0,1/2^n]$, define $L^i(F):(0,1)\times \Omega \to [0,1]$ as
\begin{align}
\label{L^i}
L^i_t(F)(\omega):=\sup \{ h\in F: [t,t+h] \in \mathcal{J}_i^n(\omega)\} \vee 0  \, ,
\end{align} 
and set $L^i:=L^i({(0,1/2^n]})$. 
We now need to prove that $L^1$ is measurable. Notice that, since $Y^{h}_t$ is continuous in $h$, $L^1=\sup_{k\geq n} L^1({D_k \cap (0,1/2^n]})$. As we proved $E$ and $Y^h_{\cdot}$ are measurable, and so such is $$A_h^1:=\{ (t,\omega)\in E: t<1-h \, , \, Y^h_t>1-\epsilon \} . $$
Since $L^1({D_k \cap (0,1/2^n]})=i/2^k$ on 
 $A^1_{i/2^k} \setminus (\cup_{j=i+1}^{2^{k-n}} A^1_{j/2^k} )$ for $i=1,\ldots , 2^{k-n}$ and $L^1({D_k \cap (0,1/2^n]})=0$ otherwise,  $L^1({D_k \cap (0,1/2^n]})$  and thus $L^1$ are measurable.
 
In Step 4 we will build random times $\mathfrak{t}_1^n, \mathfrak{h}_1^n$ from $L^1$. From $\mathfrak{t}_1^n, \mathfrak{h}_1^n$  one can define $\mathcal{J}_2^n$ as done after equation \eqref{setC} and thus $L^{2}$ as specified in \eqref{L^i}. One can then iterate the above procedure and define 
$\mathfrak{t}_i^n, \mathfrak{h}_i^n$ by induction on $i\geq 1$: from a measurable $L^{2}$ build random times $\mathfrak{t}_2^n, \mathfrak{h}_2^n$ as explained in Step 4, and from them build $\mathcal{J}_3^n$ and a measurable $L^{3}$ etc. For this to work we need to show that $L^{i}$ built from $ \mathcal{J}_i^n$ (and thus from $ \mathcal{J}_1^n$ and the random times $(\mathfrak{t}_j^n, \mathfrak{h}_j^n)_{j=1}^{i-1}$) is measurable for all $i\geq 2$; we now do so for $i=2$, the general case being only notationally more complicated. 
If $F\subseteq (0,1/2^n]$,  from the definition of $\mathcal{J}_2^n$ it follows that $L^2_t(F)$ equals  $$\bar{L}^2_t(F):=\sup \{ h\in F: [t,t+h] \in \mathcal{J}^n \text{ and } [t,t+h] \cap [\mathfrak{t}_1^n, \mathfrak{t}_1^n + \mathfrak{h}_1^n] = \emptyset \} \vee 0  \, $$
on $C^n_1$,  whereas on $\Omega \setminus C^n_1$ we have $L^2_t(F)=0$. 
 Since $\mathfrak{t}_1^n , \mathfrak{h}_1^n$ are random times and 
$ \{\mathcal{J}^n_1 \neq \emptyset \}$ is $\cF$-measurable (as proved\footnote{In Step 4 we use the measurability of $L^i$ to prove that $ \{\mathcal{J}^n_i \neq \emptyset \}=\{ S^i>0\}$ is $\cF$-measurable.} in Step 4), $C^n_1$ is $\cF$-measurable; thus, it is enough to prove that $\bar{L}^2_t(F)$ is measurable. The proof is basically the same as for $L^1$: since $A^1_h $ is measurable, such is $A^2_h:=A^1_h \cap B^1_h$  where
$$B^1_h:=  \{ (t,\omega)\in E: t> \mathfrak{t}_1^n + \mathfrak{h}_1^n \}
 \cup   \{ (t,\omega)\in E: t+h<\mathfrak{t}_1^n  \}   , $$
and since  $\bar{L}^2({D_k \cap (0,1/2^n]})=i/2^k$ on 
 $A^2_{i/2^k} \setminus (\cup_{j=i+1}^{2^{k-n}} A^2_{j/2^k} )$ for $i=1,\ldots , 2^{k-n}$ and $\bar{L}^2({D_k \cap (0,1/2^n]})=0$ otherwise,  $\bar{L}^2({D_k \cap (0,1/2^n]})$  is measurable and thus so is $\bar{L}^2=\sup_{k\geq n} \bar{L}^2({D_k \cap (0,1/2^n]})$.

 \textit{ Step 4.} In this step we explain how to use a measurable  $L^i$
 to build random times $\mathfrak{t}_i^n, \mathfrak{h}_i^n$ which satisfy the first 3 statements of \eqref{thprops} and 
 \begin{align}
\label{t+hSmall}
 \bP( \mathfrak{t}_i^n +  \mathfrak{h}_i^n = \infty ) \leq 1/2^{n+i} ;
\end{align} 
 since  $\bP( \Pi_{\Omega}(E) )=1$, it follows from \eqref{t+hSmall} that $\bP(V_n)\geq 1-1/2^n$, so   $(\mathfrak{t}_i^n, \mathfrak{h}_i^n)_{i\geq 1}$ satisfy \eqref{thprops}, concluding the proof. 
 Notice that, despite the fact that $(0,1)$ is uncountable,   $S^i=\sup_{t\in  (0,1)} L^i_t$ is also measurable with respect to the (complete) sigma algebra $\cF$: this  follows from\footnote{For a proof of this result one can consult \cite[Theorem A.5.10]{Bich02}.} \cite[Chapter 1, Theorem 4.14]{ReYo99} and the identity  $\{ S^i >\lambda \}=\Pi_{\Omega}(\{ L^i >\lambda \})$, which holds for all $\lambda \in \bR$.
 It follows that $\Pi_{\Omega}(\{ L^i>S^i /2\})=\{ S^i>0\}=\{\mathcal{J}^n_i \neq \emptyset \}$ is $\cF$-measurable, and since $L^i$ is measurable  we can apply  the section theorem (with the constant filtration $\cF_t:=\cF$) to $\{ L^i>S^i /2\}$  and obtain a random time $\bar{\mathfrak{t}}_i^n$ such that $\bP(\bar{\mathfrak{t}}_i^n=\infty \text{ and } \mathcal{J}^n_i \neq \emptyset ) \leq  1 /2^{n+i+1}$ and  $L^i_{\bar{\mathfrak{t}}_1^n}>S^i/2$ (and in particular $\mathcal{J}^n_i \neq \emptyset$ and $\bar{\mathfrak{t}}_i^n>0$) on $\{ \bar{\mathfrak{t}}_i^n<\infty \}$. 
We then   define $\mathfrak{t}_i^n$ to equal $\bar{\mathfrak{t}}_i^n$ on $\{ \mathcal{J}^n_i \neq  \emptyset \}$ and to equal $0$ otherwise.  In particular $\{\bar{\mathfrak{t}}_i^n=\infty \} \cap \{ \mathcal{J}^n_i \neq \emptyset \}=\{\mathfrak{t}_i^n =\infty\}$ and so 
$\bP(\mathfrak{t}_i^n=\infty ) \leq  1 /2^{n+i+1}$, 
 on $\{ \mathfrak{t}_i^n \in (0,\infty) \}$ we have $\mathfrak{t}_i^n= \bar{\mathfrak{t}}_i^n$,  $L^i_{\mathfrak{t}_1^n}>S^i/2$ and $\{ \mathcal{J}^n_i \neq  \emptyset \}$, and finally  $\{\mathfrak{t}_i^n= 0\} =\{ \mathcal{J}^n_i =  \emptyset \}$. Define
$$G_i:=\{(h, \omega)\in (0,1/2^n] \times \Omega  :  \mathfrak{t}_i^n \in (0,\infty) , \,   [\mathfrak{t}_i^n,\mathfrak{t}_i^n+h] \in \mathcal{J}^n_i , \,  h> S^i/2  \} \, , $$
which is measurable since $S^i, \mathfrak{t}_i^n$ and $(\mathfrak{t}_j^n,\mathfrak{h}_j^n )_{j=1}^{i-1}$ are $\cF$-measurable and $Y^{\cdot}_t$ is measurable.
Since
$ \{\mathfrak{t}_i^n \in (0,\infty) \}  =\Pi_{\Omega}(G_i)$,  by applying the section theorem to $G_i$ we find a random time $\bar{\mathfrak{h}}_i^n$ such that 
$\bP( \bar{\mathfrak{h}}_i^n=\infty \text{ and } \mathfrak{t}_i^n \in (0,\infty)) \leq 1/2^{n+i+1} $
 and  $(\bar{\mathfrak{h}}_i^n(\omega), \omega)\in G_i$   for $\omega \in 
\{\bar{\mathfrak{h}}^n_i<\infty\}$.
  We then   define $\mathfrak{h}_i^n$ to equal $0$ on  $\{ \mathfrak{t}_i^n =0\}$ and to equal $\bar{\mathfrak{h}}_i^n$ otherwise. 
 In particular $\{\mathfrak{h}_i^n =\infty\} \cap \{ \mathfrak{t}_i^n <\infty\}=\{\bar{\mathfrak{h}}_i^n=\infty \} \cap \{ \mathfrak{t}_i^n \in (0,\infty ) \}$ has probability at most $1 /2^{n+i+1}$, so  \eqref{t+hSmall} holds.
Notice that $\{\mathfrak{h}_i^n= 0\} = \{ \mathfrak{t}_i^n = 0 \}=\{ \mathcal{J}^n_i =  \emptyset \}$, and on $\{ \mathfrak{t}_i^n \in (0,\infty), \mathfrak{h}_i^n<\infty \}=C^n_i$ we have $ \bar{\mathfrak{h}}_i^n=\mathfrak{h}_i^n<\infty$ and so  $ \mathfrak{h}_i^n >S^i/2>0$ and $[\mathfrak{t}_i^n,\mathfrak{t}_i^n+h] \in \mathcal{J}^n_i$; thus we have defined random times  $\mathfrak{t}_i^n, \mathfrak{h}_i^n$ which satisfy \eqref{t+hSmall} and the first 3 statements of \eqref{thprops}, concluding the proof.

 \end{proof} 
We now strengthen the previous result as to make the quadratic variation to be exploding in all (non-trivial) intervals simultaneously.
\begin{lemma}
\label{InfQV}
There exist $\pi_n\in \cP[0,1]$ such that 
\[ \lim_n \langle B\rangle^{\pi_n}_{(\s,\t]}\to\8 \text{ a.s. on } \{\s<\t\}  \text{ as  } n\to \infty  \,  \text{ for all random times    }   0\leq \s\leq \t\leq 1 \,  .\]
\end{lemma} 
 Note that the previous equality is required to hold only on $\{\s<\t\} $ since trivially $\langle B\rangle^{\pi}_{(\s,\t]}=0$ on $\{\s=\t\} $.
\begin{proof}
Lemma \ref{StrongLevy} gives for each $n\in \bN\setminus \{0\}$ and $i=0,1,\ldots, 2^{n-1}$ some $\pi^{i}_n\in \cP[i/2^n,(i+1)/2^n]$ such that 
$ \bP(\langle B\rangle_{\pi^{i}_n} \leq 2^{n})\leq 2^{-n} $. 
Define $\pi_n:=\cup_{i=0}^{2^n-1}\pi_n^{i}\in \cP[0,1]$. Now given $\s\leq \t$ take $\o\in \{ \t -\s>2/2^n \} $, so there exists $i=i(\o)$ such that 
\[ [i/2^n,(i+1)/2^n](\o)\subseteq [\s,\t](\o)  , \]
and hence, using  \eqref{QVincreas},  we see that 
\[ \langle B\rangle^{\pi_n}_{[i/2^n,(i+1)/2^n]}(\o) \leq \langle B\rangle^{\pi_n}_{(\s,\t]}(\o) \]
and since $\pi_n \cap  [i/2^n,(i+1)/2^n]=\pi_n^{i}\cap  [i/2^n,(i+1)/2^n]$ we get that 
\[ \langle B\rangle^{\pi_n}_{[i/2^n,(i+1)/2^n]}(\o) = \langle B\rangle^{\pi_n^{i}}_{[i/2^n,(i+1)/2^n]}(\o)=\langle B\rangle_{\pi^{i}_n}(\o) , \]
Putting everything together we get that 
\[ \{ \langle B\rangle^{\pi_n}_{(\s,\t]} \leq 2^{n} \,  \text{ and   }  \,  \t -\s>2/2^n  \} \subseteq \cup_{i=0}^{2^n-1} \{ \langle B\rangle_{\pi^{i}_n}(\o) \leq 2^n \,  \text{ and   }  \,  \t -\s>2/2^n \}  , \]
and so the Borel-Cantelli lemma gives the result.
\end{proof} 
 
The following lemma states in probabilistic terms the fact that quadratic variation along the partition $\pi^{d} $ (resp. $\pi^{u}$) is only slightly bigger than  $0$ (resp. $1$).
 \begin{lemma}
\label{B/Z}
Given random times $0\leq \s\leq \t\leq 1$, $\pi\in \cP\st$ and 
a random variable $Z$ with values in $\bN\setminus \{0\}$, there exists
 $\pi'\in \cP\st$ such that $\pi\subseteq \pi'$ and 
\[ \langle B\rangle_{\pi'}= \langle B\rangle_{\pi}/Z.\]
In particular for any $\e>0$  there exist  $\pi^{d}, \pi^{u}\in \cP\st$ such that $ \pi^{d} \supseteq \pi \subseteq \pi^u$,
\[ \bP\left(\langle B\rangle_{\pi^d} >\frac{1}{Z}\right)<\e  \quad  \text{and} \quad \bP\left(\langle B\rangle_{\pi^u} \notin \left[1,1+\frac{1}{Z}\right] \text{ and } \s<\t \right) <\e .\]

\end{lemma} 
 \begin{proof}
 By working separately on each subinterval $[\s_i,\s_{i+1}]$ of $\pi=(\s_i)_{i}$, to find $\pi'$ we can assume w.l.o.g. that $\s_0=\s, \s_{i+1}=\t$ for all $i\in\bN$. Define 
\[\s'_i:=\min\{ t\geq \sigma : B_t=\frac{(B_{\t}-B_{\s})(i\wedge Z)}{Z} +B_{\s}  \} \quad \text{ on } \{ B_{\t} \neq B_{\s}\} \]
  and $\s'_0:=\s,\s'_{i+1}:=\t$ on $\{ B_{\t} = B_{\s}\} $; then $\pi\subseteq \pi':=(\s'_i)_{i\in \bN}\in \cP\st$ and 
  \[ \langle B\rangle_{\pi'}= \sum_i (B_{\s'_{i+1}} - B_{\s'_i})^2 =\left(\frac{B_{\t}-B_{\s}}{Z} \right)^2 Z = \langle B\rangle_{\pi}/Z.\]
  Now fix $\e>0$, and apply the previous result to find some $\pi'=\pi'_n$ such that \[ \langle B\rangle_{\pi'_n}= \langle B\rangle_{\pi}/n \to 0; \] taking $\pi^d:=\pi'_n$ for $n$ big enough shows that $\bP(\langle B\rangle_{\pi^d}>\frac{1}{Z})<\e$. \\  
Finally let $\pi_n$ be as in Lemma \ref{InfQV} and let $\pi'_n \in\cP\st$  be such that  $\pi'_n\supseteq \pi_n$  and  \[ \langle B\rangle_{\pi'_n}= \langle B\rangle_{\pi_n}/Y_n \quad \text{ where } \quad  Y_n:=\max\{k\in \bN: k\leq \langle B\rangle_{\pi_n}\} \vee 1 . \]
Notice that $\langle B\rangle_{\pi'_n} \leq  1+1/Y_n$ and on $\langle B\rangle_{\pi_n} \geq 1$ we have that $\langle B\rangle_{\pi'_n} \geq 1$; moreover  
 $Y_n \geq \langle B\rangle_{\pi_n} -1 \to \8$ a.s. on $\{\s<\t\}$. Taking $\pi^u:=\pi'_n$ for $n$ big enough it follows that $\bP\left(\langle B\rangle_{\pi^u} \notin \left[1,1+\frac{1}{Z}\right] \text{ and } \s<\t \right) <\e $.
\end{proof} 
  
We now essentially prove the convergence at any fixed time. 
 
\begin{lemma}
\label{ConToY}
Given random times $0\leq \s\leq \t\leq 1$ and 
a random variable $Y$ with values in $[0,\8]$, there exist
 $\pi_n\in \cP\st$ such that 
\[   \langle B\rangle^{\pi_n}_{(\s,\t]} \to Y \text{ a.s. on } \{\s<\t\}  \text{ as  } n\to \infty .\]

\end{lemma}

\begin{proof}
 On $ \{ \s=\t\}$ we define $\pi_n=\{\s\}$ for all $n$, and on $ \{Y=\8 ,\, \s<\t\}$ we set $\pi_n$ equal to the random partition $\pi_n$ given by lemma \ref{InfQV}; this clearly gives the thesis on 
$ \{Y=\8\} \cup\{ \s=\t\}$. To conclude we will define $\pi_n$ separately on each $\{ Y\in [i/2^n,(i+1)/2^n), \, \s <\t \}, i\in \bN$. We want to   build  $\tilde{\pi}_n \in  \cP\st $ such that for all $i\in \bN\setminus \{ 0\}$,
\begin{align}
\label{pitildeEQ}
\bP\left(\langle B\rangle_{\tilde{\pi}_n} \notin [i, i+1 ) , Y\in [i/2^n, (i+1)/2^n) \text{ and } \s<\t \right) < 1/2^{n+i} \, ;
\end{align} 
then if we define $\pi_n$ to be for $i\in \bN\setminus \{0\}$ the random partition $\pi'$ given by Lemma \ref{B/Z} with $Z=2^n$, and for $i=0$  the random partition $\pi^d$ given by Lemma \ref{B/Z} with $\e=1/2^{n}$ and $Z=2^n$, since trivially
\begin{align*}
\label{}
\{  | \langle B\rangle_{\pi_n}-Y| >1/2^n , Y\in [i/2^n, (i+1)/2^n) \} \subseteq 
\{   \langle B\rangle_{\pi_n} \notin [i/2^n, (i+1)/2^n) \ni Y \} \, ,
\end{align*} 
it follows that 
\[ \bP( | \langle B\rangle_{\pi_n}-Y| >1/2^n , Y<\8 \text{ and } \s<\t) < 1/2^n +\sum_{i=1}^{\8} 1/2^{n+i}=2/2^n , \]
and so Borel-Cantelli's lemma  yields the thesis.

We will construct such $\tilde{\pi}_n  $ separately on each $\{ Y\in [i/2^n,(i+1)/2^n), \, \s <\t \}, i\in \bN$ as the union over $k=0,\ldots ,i-1$ of some partitions $ \pi^{i,k}_n$
of  some subintervals $[\s^{i,k}, \s^{i,k+1}]$. First,  we  define the random times 
\[ \s^{i,k}:= \frac{\t-\s}{i} k+ \s ,  \quad i\in \bN\setminus \{0\}, k=0,\ldots , i\] 
and notice that $ \s^{i,k} < \s^{i,k+1}$ on $\{\s<\t\}$, so we can use Lemma \ref{B/Z} to find $\pi^{i,k}_n\in \cP[\s^{i,k}, \s^{i,k+1}]$ such that 
\begin{align*}
\label{}
 \bP\left(  \langle B\rangle_{\pi^{i,k}_n}  \notin \left[1,1+\frac{1}{Z_n}\right) \text{ and } \s<\t \right) < \frac{i}{2^{n+i}} \, , 
\end{align*} 
where we take $Z_n:=j$ on $\{ Y\in [j/2^n, (j+1)/2^n) \}$ for $j\in \bN\setminus\{0\}$ and\footnote{On $\{ Y\in [0, 1/2^n) \}$ we can define $Z_n$ arbitrarily, as long as it takes values in $\bN\setminus \{0\}$.} $Z_n:=1$ on $\{ Y\in [0, 1/2^n) \}$.
Intersecting with $\{ Y\in [i/2^n, (i+1)/2^n)\}$ shows in particular that
\begin{align}
\label{piikn}
\bP\left(\langle B\rangle_{\pi^{i,k}_n} \notin [1, 1+1/i ) , Y\in [i/2^n, (i+1)/2^n) \text{ and } \s<\t \right) < \frac{i}{2^{n+i}} \, .
\end{align} 
  Then we define $\tilde{\pi}_n^i:=\cup_{k=0}^{i-1} \pi^{i,k}_n$, which belongs to $ \cP[\min_k \s^{i,k},\max_k \s^{i,k+1}] =\cP[\s,\t]$, and  we 
set
\[ \tilde{\pi}_n: = \tilde{\pi}_n^i \quad \text{ on } \{ Y\in [i/2^n, (i+1)/2^n)  \text{ and }  \s<\t \} , \quad i\in \bN\setminus \{0\}\]
 and  $\tilde{\pi}_n:=\{\s\} \cup \{\t\}$ on $ \{Y\in [0, 1/2^n)\} \cup \{Y=\8 \} \cup \{\s=\t\}$, so trivially $ \tilde{\pi}_n \in  \cP\st $.
 Since  \eqref{QVadditive} gives that for $i\in \bN\setminus \{0\}$
 \[ \langle B\rangle_{\tilde{\pi}_n}  = \sum_{k=0}^{i-1}\langle B\rangle_{\pi^{i,k}_n}  \quad \text{ on } \{ Y\in [i/2^n, (i+1)/2^n) \text{ and }   \s<\t \}  \, \]
 and since $\sum_{k=0}^{i-1} a_k \notin [i,i+1)$  implies that $ a_k \notin [1,1+1/i)$ for some $k$, from \eqref{piikn} summing over $k$ and majorizing  we obtain that
 $\tilde{\pi}_n $  satisfies \eqref{pitildeEQ} for all $i\in \bN\setminus \{ 0\}$, concluding the proof.
\end{proof}

To deal with the fact that $A$ may take the value $\infty$, we have decided to work with the distance $d(a,b)=|\exp(-a)-\exp(-b)|$ on $[0,\infty]$, which is not invariant by translations yet satisfies the following property.
\begin{lemma}
\label{Triang}
Given $a_i,b_i\in [0,\infty], i=1,\ldots, n$, we have
 \begin{align}
\label{ }
d\left(\sum_{i=0}^n a_i, \sum_{i=0}^n b_i \right) \leq \sum_{i=0}^n d(a_i,b_i)  \, .
\end{align} 
\end{lemma} 
\begin{proof}
It is enough to prove this for $n=2$ and then iterate. Since $\exp(-t)$ is positive decreasing  if $0\leq b\leq a\leq v\leq \infty$ and $u\geq 0$ 
from $d(a,b)= \int_b^a \exp(-t) dt$ we obtain 
\begin{align}
\label{ineqd}
d(a+u,b+u)\leq d(a,b) \text{ and } d(a,b)\leq d(v,b) \, .
\end{align} 
Assume w.l.o.g. that $a_1\geq b_1$; then  we claim that
\[ \text{ if }  a_2<b_2  \text{ then } \quad d(a_1+a_2,b_1+b_2)\leq \max(d(a_1,b_1), d(a_2,b_2)) \, : \]
 indeed if $a_1+a_2\geq b_1+b_2$ using \eqref{ineqd} one has
  $$d(a_1+a_2,b_1+b_2)\leq d(a_1+b_2,b_1+b_2) \leq d(a_1,b_1) \, ,$$ and the other case is analogous. If  $a_2 \geq b_2$ by the linearity of the integral and  \eqref{ineqd} 
  $$ d(a_1+a_2,b_1+b_2)=
   d(a_1+a_2,a_1+b_2)+ d(a_1+b_2,b_1+b_2)\leq  d(a_2,b_2) + d(a_1,b_1) .$$
\end{proof}

We can finally stitch all the pieces together. We define $\infty -\infty:=0$, so in the following proof the quantities $ A_{t_{i+1}} -A_{t_i}$ (with $t_i \leq t_{i+1}$)  are always well defined and satisfy $\sum_{i=0}^{n-1} A_{t_{i+1}} -A_{t_i}= A_{t_{n}} -A_{t_0}$. 

\begin{proof}[Proof of Theorem \ref{ArbQV}] 
Thanks to Scholium \ref{Time01} it is enough to work on the time interval $[0,1]$. For  simplicity,  we will first build (in step 1 and 2) $\pi_n$ which may fail to be refining and to include $\bar{\pi}_n$  but does satisfy the other assertions of the theorem. 

\textit{Step 1.}
To isolate the main idea from the technicalities we first deal with the case of  continuous $A$, using the same notation as for the general case so as to be able to refer back to this case  later;  denote by $\tilde{\pi}_n=(\s_n^i)_{i\in\bN}$ the `random' partition 
$\cup_{i=0}^{2^n} \{ i/2^n\}$ and notice that  $\s_n^i=1$ if $i\geq \tilde{i}:=2^n$, so it is enough to consider from now on $i\leq \tilde{i}-1$.
 Given  $\pi \in \cP[\s_n^i, \s_n^{i+1}]$
 define \begin{align}
\label{alongD_n}
\Delta_n^i(\pi)=\Delta_n^i(\pi,A):=d( \langle B\rangle^{\pi}_{[\s_n^i,\s_n^{i+1}]} \,  , A_{\s_n^{i+1}} -A_{\s_n^i} )
\end{align} 
Notice that by definition $\Delta_n^i(\pi,A)=0$ on $\{\s_n^i=\s_n^{i+1}\}$; thus, thanks to Lemma \ref{ConToY},  we can find $\pi_n^i \in \cP[\s_n^i, \s_n^{i+1}]$ such that 
 \[ \bP(\Delta_n^i(\pi_n^i)> 2^{-n}/ \tilde{i})< 2^{-n}/ \tilde{i}. \]
 Setting $\pi_n:=\cup_{i=0}^{\tilde{i}-1} \pi_n^i$ and using \eqref{TriangleIneqfor<B>} and 
 Lemma \ref{Triang} we get, writing $\langle B\rangle^{\pi_n}_t$ and $A_t$ as the sums of their increments over the subintervals of $\tilde{\pi}_n$
 $$ \textstyle  d( \langle B\rangle^{\pi_n}_{\s_n^k} , A_{\s_n^k} )  \leq \sum_{i=0}^{k-1} \Delta_n^i(\pi_n^i)  \leq \sum_{i=0}^{\tilde{i}-1} \Delta_n^i(\pi_n^i) \, ;$$ 
if the sum over $\tilde{i}$ positive terms is greater than $2^{-n}$ then at least one summand is greater than  $2^{-n}/ \tilde{i}$ and so 
\begin{align}
\label{MaxTriangle}
\{ \max_{t\in \tilde{\pi}_n} d( \langle B\rangle^{\pi_n}_t , A_t ) >1/2^n  \} \subseteq 
\cup_{i=0}^{\tilde{i}-1}  \{  \Delta_n^i(\pi_n^i) > 2^{-n}/ \tilde{i} \, \} \, ,
\end{align} 
and so  we  obtain that for $k=n$
\begin{align}
\label{D_kNotYetThere}
\textstyle
\bP( \max_{t\in \tilde{\pi}_k} d( \langle B\rangle^{\pi_n}_t , A_t ) >1/2^n ) \leq 1/2^n  .
\end{align}  
Since $(\tilde{\pi}_k)_k$ is refining 
\[ \max_{t\in \tilde{\pi}_k} d( \langle B\rangle^{\pi_n}_t , A_t )
\leq 
\max_{t\in \tilde{\pi}_n} d( \langle B\rangle^{\pi_n}_t , A_t )   \, , \]
and so \eqref{D_kNotYetThere} holds for all $k \leq n$; 
thus for each fixed $k$ we can  apply Borel-Cantelli's lemma to get that $  \max_{t\in \tilde{\pi}_k} d( \langle B\rangle^{\pi_n}_t , A_t )\to 0$ a.s. as $n\to\8$, and  since $F:=\cup_{k\in \bN} \tilde{\pi}_k$
 is dense in $[0,1]$ and contains the  set of jumps of $A$ (which in the case of step 1 is empty), Lemma \ref{ConvAtJumps} gives\footnote{As $[0,\8]$  is homeomorphic to $[0,1]$,  Lemma \ref{ConvAtJumps} holds if $a$ has values in $[0,\infty]$ instead of $[0,1]$.}
  the convergence for all $t\in [0,1]$, uniformly over every interval where $A$ is continuous (which in the case of step 1 is everywhere). Notice that since $\tilde{\pi}_n\supseteq \bN / 2^n$ by construction $\tau_n^i +  2^{-n}$ is a stopping time, where $\pi_n=(\tau_n^i)_i$.
 
\textit{Step 2.}
We deal now with a general increasing process $A$. Consider the positive increasing   process $D_t:=\lim_{s\in \bQ, s\downarrow t} 1-\exp(-A_s)$, which is \emph{c{\`a}dl{\`a}g, bounded by} $1$, and has the same times of jump as $A$.
Construct $\pi'_n$ by setting $\t^0_n:=0$,
\begin{align}
\label{PiJumps}
\t^{i+1}_{n}:=\inf\{t>\t^i_n : D_t - \lim_{s\uparrow t} D_s > 1/n \}\wedge 1 \, , \quad \text{ and   }  \, \pi'_n:=(\t_n^i)_i .
\end{align} 
Notice that the $\t_n^i$ are random times 
and $\pi'_n$ are random partitions of $[0,1]$ (for an elementary proof see\footnote{The  cited lemma deals with stopping times and c{\`a}dl{\`a}g adapted processes; these reduce to random times  and c{\`a}dl{\`a}g processes when considering a constant filtration.}  \cite[Lemma 3.3]{Sio13DM}) and $\cup_n \pi_n'$ contains all the times of jumps of $D$, i.e. of $A$.
 Since $D$ is increasing and $D_1\leq 1$ we have that $\t^i_n=1$ for any $i\geq  n$. Now define $\tilde{\pi}_n=(\s_n^i)_{i\in\bN}\in \cP[0,1]$ as the random partition $ \pi'_n \cup (\cup_{i=0}^{2^n} \{ i/2^n\})$, notice that   $\s_n^i=1$ for any $i\geq \tilde{i}:=n+2^n$  and that  $(\tilde{\pi}_n)_n$ is refining. 
The proof given for continuous $A$ then applies word by word, giving \eqref{D_kNotYetThere} and the thesis.

\textit{Step 3.}
Finally, we will now improve on the above proof and show that $\pi_n$ can be chosen to be refining and to include $\bar{\pi}_n$.
We will define $(\pi_n, \tilde{\pi}_n)_n$ by induction; more precisely we set $\pi_0:=\tilde{\pi}_0:=\bar{\pi}_0$,   and for $n\geq 1$ we will define  $ \tilde{\pi}_n$ given $(\pi_k)_{k< n}$ and then define  $\pi_n$  given  $ \tilde{\pi}_n$.
Let $ \pi'_n$ be as in \eqref{PiJumps}, set
\[ \tilde{\pi}_n:= \pi'_n \cup (\cup_{i=0}^{2^n} \{ i/2^n\})\cup \pi_{n-1} \cup (\cup_{k=0}^{n} \bar{\pi}_{k}) ,
\] 
and notice that    $(\tilde{\pi}_k)_{k\leq n}$ is refining (since $(\pi_k)_{k\leq n-1}$ is refining, by inductive hypothesis).
We now endeavor to construct some $\pi_n \supseteq \tilde{\pi}_n=(\s_n^i)_{i\in\bN}$ such that \eqref{D_kNotYetThere} holds for $k=n$, which 
 would imply \eqref{<B>ConvA} (as in step 1), and since $\pi_{n-1}\cup \bar{\pi}_n \subseteq \tilde{\pi  }_n  \subseteq \pi_n$ the proof would be over.
We will now make use of the random variables $K(\tilde{\pi}_n)$ and $\Delta_n^i(\pi,A)$ defined in \eqref{Kpi} and \eqref{alongD_n}.  Thanks to Lemma \ref{ConToY} for each $\tilde{i }\in \bN\setminus \{0\}$ such that $ \bP(K(\tilde{\pi}_n)=\tilde{i })>0$ 
there exists $\pi^{i,\tilde{i }}_n\in  \cP[\s_n^i, \s_n^{i+1}]$ such that 
\[\bP(\Delta_n^i(\pi^{i,\tilde{i }}_n, A)> 1/2^n \tilde{i})<  \bP(K(\tilde{\pi}_n)=\tilde{i })/2^n \tilde{i} . \]
 Then we set  $\pi_{n}^i:=\pi^{i,\tilde{i }}_n$ on $\{ K(\tilde{\pi}_n)=\tilde{i } \}$ for each $\tilde{i }$ such that $ \bP(K(\tilde{\pi}_n)=\tilde{i })>0$; this defines $\pi_{n}^i$ on a set of full measure, and on its complement we can define $\pi_{n}^i:=\{\s_n^i \}\cup \{ \s_n^{i+1} \} $. Then $\pi_{n}^i$ belongs to $ \cP[\s_n^i, \s_n^{i+1}]$ and for every $\tilde{i }\in \bN \setminus \{0\}$  
 \begin{align}
\label{DeltaKi}
\bP(\Delta_n^i(\pi^{i}_n, A\wedge n)> 1/2^n \tilde{i}  \,  \text{  and   }  K(\tilde{\pi}_n)=\tilde{i }   \, ) \leq  \bP(K(\tilde{\pi}_n)=\tilde{i })/2^n \tilde{i} \, . 
\end{align} 
Now we set $\pi_n:=\cup_{i\in \bN} \pi_n^i$ and notice that $\pi_n$  on $\{ K(\tilde{\pi}_n)=\tilde{i } \}$ equals\footnote{Indeed for $i\geq \tilde{i }$ on $\{ K(\tilde{\pi}_n)=\tilde{i } \}$ we have $\s_n^i=1=\s_n^{i+1}$ and so  $\pi_{n}^i=\{1\}\subseteq \pi_{n}^{\tilde{i } -1}$.} the \emph{finite }  union $ \cup_{i=0}^{\tilde{i}-1} \pi^{i}_n$, and so by the same argument as for \eqref{MaxTriangle} we get that
\[ M_n^{\tilde{i }}:= \{ \max_{t\in \tilde{\pi}_n} d( \langle B\rangle^{\pi_n}_t , A_t) >1/2^n   \text{  and   }  K(\tilde{\pi}_n)=\tilde{i }    \}  \]
is a subset of 
\[ \cup_{i=0}^{\tilde{i}-1} \,\,  \{  \Delta_n^i(\pi_n^i, A) > 1/2^n \tilde{i}    \text{  and   }  K(\tilde{\pi}_n)=\tilde{i }   \}  \, , \]
and thus \eqref{DeltaKi} shows that $M_n^{\tilde{i }}$ has probability smaller than $\bP(K(\tilde{\pi}_n)=\tilde{i })/2^n  $. The proof is concluded since
\[
 \bP( \max_{t\in \tilde{\pi}_n} d( \langle B\rangle^{\pi_n}_t , A_t ) >1/2^n ) =  \sum_{\tilde{i }\in \bN}\bP(M_n^{\tilde{i }}) \leq \sum_{\tilde{i }\in \bN} \bP(K(\tilde{\pi}_n)=\tilde{i })/2^n  = 1/2^n \, . \]
As before  our construction gives that $\tau_n^i+  2^{-n}$ is a stopping time as $\tilde{\pi}_n\supseteq \bN / 2^n$.
\end{proof}


\end{document}